\documentclass[reqno,10pt]{amsart}

\usepackage{amssymb,psfrag,graphicx}
\usepackage[utf8]{inputenc}
\usepackage{enumerate}

\numberwithin{equation}{section}

\theoremstyle{plain}% default
\newtheorem{thm}{Theorem}[section]

\newtheorem{prop}{Proposition}[section]

\theoremstyle{definition}
\newtheorem{rmk}{Remark}[section]

\newcommand{\be}{\begin{equation}}
\newcommand{\ee}{\end{equation}}
\newcommand{\bea}{\begin{eqnarray}}
\newcommand{\eea}{\end{eqnarray}}
\newcommand{\beas}{\begin{eqnarray*}}
\newcommand{\eeas}{\end{eqnarray*}}
\newcommand{\rd}{{\text{\rm d}}}
\DeclareMathOperator{\Exp}{e}
\newcommand{\LIM}{\textsc{Lim}}
\newcommand{\f}{\mathbf{f}}
\newcommand{\bh}{\mathbf{h}}
\newcommand{\bu}{\mathbf{u}}
\newcommand{\bv}{\mathbf{v}}
\newcommand{\bw}{\mathbf{w}}
\newcommand{\bx}{\mathbf{x}}

\newcommand{\bg}{\mathbf{g}}
\newcommand{\mA}{\mathcal{A}}
\newcommand{\mB}{\mathcal{B}}
\newcommand{\mC}{\mathcal{C}}

\newcommand{\mU}{\mathcal{U}}
\newcommand{\mV}{\mathcal{V}}

\begin{document}

\title[A Discrete data assimilation scheme for the 2D NSE and statistics]{A discrete data assimilation scheme for the solutions of the 2D Navier-Stokes equations and their statistics}

\date{\today}

\author[C. Foias] {Ciprian Foias}
\address{ \textnormal{(Ciprian Foias)} Department of Mathematics\\
Texas A\&M University\\ College Station, TX 77843, USA.}
\email[C. Foias]{foias@math.tamu.edu}

\author[C. F. Mondaini]{Cecilia F. Mondaini}
\address{\textnormal{(Cecilia F. Mondaini)} Department of Mathematics\\
Texas A\&M University\\ College Station, TX 77843, USA.}
\email[C. F. Mondaini] {cfmondaini@math.tamu.edu}

\author[E. S. Titi]{Edriss S. Titi}
\address{\textnormal{(Edriss S. Titi)} Department of Mathematics\\
Texas A\&M University\\ College Station, TX 77843, USA. {\bf ALSO}, Department of Computer Science and Applied Mathematics, Weizmann Institute of Science, Rehovot 76100, Israel.}
\email[E. Titi] {titi@math.tamu.edu and edriss.titi@weizmann.ac.il}

%\address{$\dagger$ to whom correspondence should be made}

\begin{abstract}
We adapt a previously introduced continuous in time data assimilation (downscaling) algorithm for the 2D Navier-Stokes equations to the more realistic case when the measurements are obtained discretely in time and may be contaminated by systematic errors. Our algorithm is designed to work with a general class of observables, such as low Fourier modes and local spatial averages over finite volume elements. Under suitable conditions on the relaxation (nudging) parameter, the spatial mesh resolution and the time step between successive measurements, we obtain an asymptotic in time estimate of the difference between the approximating solution and the unknown reference solution corresponding to the measurements, in an appropriate norm, which shows exponential convergence up to a term which depends on the size of the errors. A stationary statistical analysis of our discrete data assimilation algorithm is also provided.
\end{abstract}


\maketitle

\section{Introduction}

The idea of data assimilation is to obtain a good approximation of the state of a certain physical system by combining observational data with dynamical principles pertaining to the underlying mathematical model, which allows for a downscaling process. It is widely used in many fields of geosciences, mainly for oceanic and atmospheric forecasting.

One of the approaches to this question consists in the use of feedback control algorithms, which have been considered by many researchers in the past few decades. Among the earlier related works, see, e.g., \cite{Luenberger1971} and references therein, for a study regarding linear time-invariant systems, and also \cite{Nijmeijer2001,Thau1973}, concerning nonlinear systems.

Within this context, the classical method of continuous in time data assimilation (i.e. by employing observational measurements obtained continuously in time) consists in inserting the measurements directly into the model as it is integrated in time \cite{CharneyHalemJastrow1969,Daley1991}. For example, in the case of measurements given by projections onto low Fourier modes, one may explore this idea by introducing the low Fourier mode observables into the equation for the evolution of the high Fourier modes \cite{BrowningHenshawKreiss1998,HaydenOlsonTiti2011,HenshawKreissYstrom2003,Korn2009,OlsonTiti2003,OlsonTiti2008}. However, in the case when the measurements are collected from a discrete set of nodal points, such approach may present some difficulties, since it is not possible to compute the exact values of the spatial derivatives present in the model.

In \cite{Azouanietal2014}, a new algorithm of continuous in time data assimilation was introduced inspired by ideas from control theory \cite{AzouaniTiti2014}. This algorithm was given for the 2D Navier-Stokes equations, but it is in fact applicable to a large class of dissipative evolution equations. In this new approach, instead of inserting the measurements directly into the model, a feedback control term is introduced into the original evolution equation of the system, which forces the coarse spatial scales of the solution of the new model, i.e. the \textit{approximating solution}, towards the coarse spatial scales of the solution of the original system, i.e. the \textit{reference solution}. This type of technique was previously called Newtonian nudging or dynamic relaxation method (see, e.g. \cite{HokeAnthes1976} and references therein), and usually considered in much simpler scenarios. The advantage in this new algorithm is that no derivatives are required of the observational measurements, and thus it 
works for a general class of interpolant operators.

More specifically, consider the 2D incompressible Navier-Stokes equations, which are given by
\be\label{NSEeqs}
 \frac{\partial \bu}{\partial t} - \nu \Delta \bu + (\bu \cdot \nabla) \bu + \nabla p = \f, \quad \nabla \cdot \bu = 0.
\ee
where $\bu = (u_1,u_2)$ and $p$ are the unknowns, and represent the velocity vector field and the pressure, respectively; while $\f$ and $\nu$ are given, and represent the mass density of volume forces applied to the fluid and the kinematic viscosity parameter, respectively. 

In this case, the reference solution is given by a solution $\bu$ of \eqref{NSEeqs}, for which the initial data is missing. In \cite{Azouanietal2014}, the measurements corresponding to this reference solution are assumed to be free of errors, obtained continuously in time and discretely in space, from a mesh in the physical domain with resolution of size $h$. In order to deal with these spatial discrete measurements in the data assimilation algorithm, an interpolation operator in space is considered and denoted by $I_h$. This operator is assumed to satisfy a suitable condition of approximation of identity. Then, the continuous in time data assimilation algorithm introduced in \cite{Azouanietal2014} consists in finding a solution $\bv = \bv(\bx,t)$ of the following problem, for $(\bx,t) \in \Omega \times [t_0,\infty) \subset \mathbb{R}^2 \times \mathbb{R}$:
\be\label{introContDataAssAlg}
\frac{\partial \bv}{\partial t} - \nu \Delta \bv + (\bv \cdot \nabla) \bv + \nabla \tilde{p} = \f  - \beta(I_h(\bv) - I_h(\bu)),
\ee
\be
 \nabla \cdot \bv = 0,
\ee
\be
\bv(t_0) = \bv_0,
\ee
where $\bu = \bu(\bx,t)$ is the reference solution, satisfying \eqref{NSEeqs}; $\tilde{p}$ is a modified pressure; $\beta$ is the relaxation (nudging) parameter; and $\bv_0$ is an arbitrary initial condition. Notice that the measurements are incorporated into the algorithm through the second term on the right-hand side of \eqref{introContDataAssAlg}, the feedback control term, where the coarse spatial scales of $\bv$ are forced towards the coarse spatial scales of $\bu$ with the help of the relaxation parameter $\beta$.

The idea behind the introduction of this feedback term comes from the fact that certain dissipative evolution equations are determined by only a finite number of degrees of freedom, a result that was first rigorously proved by Foias and Prodi \cite{FoiasProdi1967} in the case of low Fourier modes for the 2D Navier-Stokes equations. More specifically, they prove that if the first $N$ Fourier modes of any two solutions tend to zero asymptotically in time, for a sufficiently large $N$, then the difference between the whole two solutions also tends to zero asymptotically in time. Later, similar results were proved for other types of degrees of freedom, such as nodes and local averages over volume elements \cite{FoiasTemam1984,FoiasTiti1991,JonesTiti1992,JonesTiti1993}, and also for a larger class of dissipative evolution equations \cite{CockburnJonesTiti1997}.

This suggests that, under suitable conditions on the associated parameters, forcing the coarse spatial scales of the approximating solution towards those of the reference solution should be enough for making them converge to each other asymptotically in time. Indeed, in \cite{Azouanietal2014} the authors show that, given an arbitrary initial data $\bv_0$ for the approximate model and under suitable conditions on the spatial resolution $h$ and on the coefficient $\beta$ of the feedback term (the relaxation parameter), the approximate solution $\bv$ converges exponentially to the reference solution $\bu$ as time goes to infinity, in an appropriate norm.

Notably, in \cite{AltafTitiKnioZhaoMcCabeHoteit2015,GeshoOlsonTiti2015}, numerical studies of the continuous data assimilation algorithm from \cite{Azouanietal2014} were done, indicating that the algorithm actually performs remarkably well for significantly less restrictive values of the parameters $\beta$ and $h$ than the ones suggested by the analytical estimates in \cite{Azouanietal2014}.

Several other works followed by using the approach from \cite{Azouanietal2014} applied to different systems of equations and also different contexts such as missing some state variables in the observations \cite{AlbanezNussenzveigTiti2014,BessaihOlsonTiti2015, FarhatJollyTiti2015, FarhatLunasinTiti2015-1, FarhatLunasinTiti2015-2, MarkowichTitiTrabelsi2015} (see also \cite{GhilHalemAtlas1978,GhilShkollerYangarber1977} for previous approaches in this subject). However, most of these previous works assumed that the measurements were free of errors and obtained continuously in time. Although, we remark that continuous in time data assimilation with stochastically noisy data was considered in \cite{BessaihOlsonTiti2015} by using the framework from \cite{Azouanietal2014}, and error-free spatio-temporal discrete data assimilation was studied in \cite{HaydenOlsonTiti2011} by using a previous algorithm. Moreover, noisy observations were also considered in \cite{BlomkerLawStuartZygalakis2013,LawShuklaStuart2014} in 
the context of the 3DVAR filtering method.

The purpose of this present work is to adapt the continuous in time data assimilation algorithm introduced in \cite{Azouanietal2014} for the 2D Navier-Stokes equations to the more realistic case when the measurements are obtained discretely in time and may be contaminated by systematic errors. 

In this context, we consider an increasing sequence of instants of time $\{t_n\}_{n\in\mathbb{N}} \subset \mathbb{R}$ at which measurements are taken, such that $t_n \to \infty$. The maximum step size between successive measurements is denoted by a positive constant $\kappa$, so that $|t_{n+1} - t_n| \leq \kappa$, for all $n \in \mathbb{N}$. Also, for each $n \in \mathbb{N}$, we denote by $\eta_n$ the error associated to the measurement at time $t_n$. Moreover, we denote as before by $h$ the resolution of the mesh in the physical domain and consider an interpolation operator in space given by $I_h$. Here, we assume that $I_h$, besides satisfying the condition of approximation of identity, is also a bounded operator with respect to an appropriate norm. Our discrete data assimilation algorithm consists in finding a solution $\bv = \bv(\bx,t)$ of the following problem, for $(\bx,t) \in \Omega \times [t_0,\infty) \subset \mathbb{R}^2 \times \mathbb{R}$:
\be\label{introDiscDataAssAlg}
\frac{\partial \bv}{\partial t} - \nu \Delta \bv + (\bv \cdot \nabla) \bv + \nabla q = \f  - \beta \sum_{n=0}^\infty (I_h(\bv(t_n)) - I_h(\bu(t_n))) \chi_n  + \beta \sum_{n=0}^{\infty} \eta_n \chi_n,
\ee
\be
 \nabla \cdot \bv = 0,
\ee
\be
\bv(t_0) = \bv_0,
\ee
where $\bu = \bu(\bx,t)$ is the reference solution, satisfying \eqref{NSEeqs}; $q$ is a modified pressure; $\beta$ is the relaxation parameter; $\chi_n$ is the characteristic function of the interval $[t_n, t_{n+1})$; and $\bv_0$ is an arbitrary initial data. 

Notice that the difference between \eqref{introContDataAssAlg} and \eqref{introDiscDataAssAlg} relies on the interpolation in time that is done in the feedback term through the characteristic functions $\chi_n$, and also on the additional last term on the right-hand side of \eqref{introDiscDataAssAlg}, which is due to the errors. In one of our main results (Theorem \ref{thmasympestwGeneral}), we show that, under suitable conditions on $\beta$, $\kappa$ and $h$, the asymptotic in time limit of the difference between the approximate solution $\bv$ satisfying \eqref{introDiscDataAssAlg} and the reference solution $\bu$ satisfying \eqref{NSEeqs} (corresponding to the error-contaminated measurements $I_h(\bu(t_n)) + \eta_n$, $n\in \mathbb{N}$) is bounded above by the maximum size of the errors multiplied by an absolute constant. In particular, our result shows that there is no accumulation of errors in time. Also, in the particular case of error-free measurements, we have exponential convergence of $\bv$ towards 
$\bu$, an analogous result to the one obtained in \cite{Azouanietal2014}.

Moreover, we also study the stationary statistical behavior of our algorithm. This is motivated by the fact that most applications of data assimilation involve fully developed turbulent flows, which display a regular statistical behavior as opposed to the highly unpredictable behavior of instantaneous values. Thus, it is common for experimentalists to consider averages (with respect to time, space or to an ensemble of experiments) of the characteristic physical quantities associated to the flow, such as energy, enstrophy, energy dissipation rate, etc., and also their correlations in space or time.

For this statistics part, assuming again the same hypotheses of the result for individual solutions, we obtain in Theorem \ref{thmasympestdifftimeav} an asymptotic in time estimate of the difference between time averages of characteristic physical quantities associated to the approximating solution and the reference solution, given by a superior bound which depends again on the size of the errors. Also, again under the same hypotheses, but considering the particular case in which the measurements are error-free, we show in Theorem \ref{thmEqualityTimeAvEnsembleAv} that the limit of time averages associated to the approximating solution coincide with the ensemble averages, with respect to a certain invariant measure, associated to the reference solution.
 
Therefore, these results allow one to obtain information on averages of physical quantities associated to the unknown reference solution through averages of the same quantities associated to the approximating solution, which can be effectively computed through our discrete data assimilation algorithm.
 
We remark that, although we consider here a discrete in time data assimilation algorithm for the 2D Navier-Stokes equations, similar ideas can also be applied to other dissipative systems, in the same way that the continuous in time data assimilation algorithm from \cite{Azouanietal2014} was extended to other equations, such as the B\'enard convection model \cite{FarhatJollyTiti2015,FarhatLunasinTiti2015-2}, the Navier-Stokes-$\alpha$ equations \cite{AlbanezNussenzveigTiti2014} and the Brinkman-Forchheimer-Extended Darcy model \cite{MarkowichTitiTrabelsi2015}.
 
This paper is organized as follows. In section \ref{sec2DNSESetting}, we review some of the mathematical setting related to the 2D incompressible Navier-Stokes equations. In section \ref{secDiscDataAssAlg}, we give further details about our discrete data assimilation algorithm and prove existence and uniqueness of weak and strong solutions. In section \ref{secAsympAnalysis}, we analyze the asymptotic behavior in time of individual solutions of our algorithm and divide it into two subsections: first, we consider the particular case of an interpolant operator given by the projection onto low Fourier modes; later, we consider the case of a general interpolant operator. The reason for considering the particular case of Fourier modes separately is that it provides more freedom in the calculations and yields slightly less strict conditions on the parameters. In section \ref{secStatAnal}, we analyze the stationary statistical behavior of our algorithm. Finally, in the Appendix, we consider the particular example of 
interpolant operator given by the sum of averages over finite volume elements, and show explicit uniform bounds of the sequence of errors $\{\eta_n\}_{n\in\mathbb{N}}$ in the $L^2(\Omega)^2$ and $H^1(\Omega)^2$ norms, with respect to given parameters.
 
\section{Mathematical setting of the 2D Navier-Stokes equations}\label{sec2DNSESetting}

In this section we provide a brief overview of the background material on the two-dimensional incompressible Navier-Stokes equations, given in \eqref{NSEeqs}. For a more detailed discussion, the reader is referred to, e.g., \cite{bookcf1988,FMRT2001,Temambook2001}. 

We denote the space variable by $\bx = (x_1,x_2)$, which varies in a domain $\Omega \subset \mathbb{R}^2$. Also, the time variable is denoted by $t$ and varies in the interval $[t_0,\infty)$, with $t_0$ representing the initial time. For simplicity, we assume that the forcing term $\f$ is time-independent and satisfies $\f \in L^2(\Omega)^2$. Notably, similar relevant results are also valid if $\f \in L^{\infty}(t_0,\infty;L^2(\Omega)^2)$.

We endow system \eqref{NSEeqs} with either periodic or no-slip boundary conditions. In the periodic case, we consider $\Omega = (0,L) \times (0,L)$ and assume that the flow is periodic with period $L$ in each spatial direction $x_i$, $i=1,2$. Moreover, we assume that
% \[
%  \int_{\Omega} \bu(\bx,t) \rd \bx = 0, \quad \int_{\Omega} \f(\bx) \rd \bx = 0.
% \]
\[
 \int_{\Omega} \f(\bx) \rd \bx = 0.
\]

In the no-slip case, we consider $\Omega$ as a bounded subset of $\mathbb{R}^2$ with sufficiently smooth boundary $\partial \Omega$ and assume that $\bu = 0$ on $\partial \Omega$.

With respect to each type of boundary condition, we have a different space of test functions, denoted by $\mV$. In the periodic case, we consider $\mV$ as the set of all $L$-periodic trigonometric polynomials from $\mathbb{R}^2$ to $\mathbb{R}^2$ that are divergence free and have zero average. In the no-slip case, we consider $\mV$ as the family of divergence free and compactly supported $C^\infty$ vector fields defined on $\Omega$ with values in $\mathbb{R}^2$. 

%With respect to each type of boundary condition, we consider a different space of test functions. In the periodic case, we have
% \[
%  \mV_{\textnormal{per}} = \left\{ \bu = \bw|_{\Omega} ; \begin{array}{l} \mbox{$\bw \in \mC^{\infty}(\mathbb{R}^2)^2$, $\nabla \cdot \bw = 0$, $\int_\Omega \bw(\bx) \rd \bx = 0$, $\bw(\bx)$ is periodic} \\ \mbox{with period $L$ in each direction $0x_i$} \end{array} \right\},
% \]
% and for the non-slip boundary condition, we consider
% \[
%  \mV_0 = \{ \bu \in \mC_c^{\infty}(\Omega)^2; \, \nabla \cdot \bu = 0\},
% \]
% where $\mC_c^{\infty}(\Omega)^2$ is the space of infinitely differentiable and real-valued functions with compact support in $\Omega$. 

% Let us denote more generally by $\mV$ the space of test functions, which may be $\mV_{\textnormal{per}}$ or $\mV_0$, depending on the case. Then, we denote by $H$ the closure of $\mV$ with respect to the norm in $L^2(\Omega)^2$, and by $V$ the closure of $\mV$ under the $H^1(\Omega)^2$ norm. Moreover, we identify $H$ with its dual space, $H'$, so that we obtain $V \subseteq H  \subseteq V'$, with the injections being continuous, compact and each space dense in the following one.

We denote by $H$ the closure of $\mV$ with respect to the norm in $L^2(\Omega)^2$, and by $V$ the closure of $\mV$ under the $H^1(\Omega)^2$ norm. Also, we denote by $H'$ and $V'$ the dual spaces of $H$ and $V$, respectively. We identify $H$ with its dual, so that we obtain $V \subseteq H  \subseteq V'$, with the injections being continuous, compact and each space dense in the following one. The duality product between $V$ and $V'$ is denoted by $\langle \cdot, \cdot \rangle_{V',V}$.

We consider the same notation from \cite{FMRT2001} and denote the inner products in $H$ and $V$ by $( \cdot,\cdot)_{L^2} $ and $(\!( \cdot, \cdot )\!)_{H^1}$, respectively. They are defined as
\[
 (\bu,\bv)_{L^2} = \int_{\Omega} \bu(\bx) \cdot \bv(\bx) \rd \bx, \quad \forall \bu,\bv \in H,
\]
\[
 (\!(\bu,\bv)\!)_{H^1} = \int_{\Omega} \sum_{i=1}^2  \frac{\partial \bu}{\partial \bx_i} \cdot \frac{\partial \bv}{\partial x_i} \rd \bx, \quad \forall \bu,\bv \in V,
\]
and the associated norms are given by $|\bu|_{L^2} = (\bu,\bu)^{1/2}_{L^2}$, $\|\bu\|_{H_1} = (\!(\bu,\bu)\!)^{1/2}_{H^1}$. 

The fact that $\|\cdot\|_{H^1}$ defines a norm in $V$ is justified via the Poincar\'e inequality, given by
\be\label{ineqPoincare}
 \lambda_1^{1/2}|\bu|_{L^2} \leq \|\bu\|_{H^1}, \quad \forall \bu \in V,
\ee
where $\lambda_1$ is the first eigenvalue of the Stokes operator, defined in \eqref{defStokesop} below.

Given any $R>0$, we denote by $B_H(R)$ and $B_V(R)$ the closed balls centered at $0$ with radius $R$ in $H$ and $V$, respectively.

Let $P_{\sigma}$ be the Leray-Helmholtz projector, which is defined as the orthogonal projection from $L^2(\Omega)^2$ onto $H$. Then, applying $P_{\sigma}$ to system \eqref{NSEeqs}, we obtain the following functional equation, which is equivalent to \eqref{NSEeqs},

\be\label{functionaleq}
 \frac{\rd \bu}{\rd t} + \nu A\bu + B(\bu,\bu) = \bg \quad \mbox{in $V'$},
\ee
where $\bg = P_{\sigma} \f \in H$, $B:V\times V \rightarrow V'$ is the bilinear operator defined as the continuous extension of the operator given by
\[
 B(\bu,\bv) = P_{\sigma}((\bu \cdot \nabla)\bv), \quad \forall \bu, \bv \in \mV,
\]
and $A: D(A) \subseteq V \rightarrow V'$ is the Stokes operator, defined as the continuous extension of
\be\label{defStokesop}
 A\bu = -P_{\sigma} \Delta \bu, \quad \forall \bu \in \mV,
\ee
with the domain of $A$, $D(A)$, given by $V\cap H^2(\Omega)^2$.

The Stokes operator is a positive and self-adjoint operator with compact inverse. Therefore, it admits an orthonormal basis of eigenvectors $\{\bw_m\}_{m\in \mathbb{N}}$ associated with a nondecreasing sequence of positive eigenvalues $\{\lambda_m\}_{m\in\mathbb{N}}$, with $\lambda_m \to \infty$ as $m \to \infty$.

We also consider, for each $m \in \mathbb{N}$, the low modes projector $P_m$, which is defined as the orthogonal projector of $H$ onto the subspace $H_m = \textnormal{span}\{ \bw_1,\ldots,\bw_m \}$. Moreover, we denote $Q_m = I - P_m$.

The bilinear operator $B$ satisfies the following property:
\be\label{PropBilTerm}
 \langle B(\bu_1,\bu_2), \bu_3\rangle_{V',V} = -\langle B(\bu_1,\bu_3), \bu_2 \rangle_{V',V}, \quad \forall \bu_1,\bu_2, \bu_3 \in V.
\ee

We recall some well-known inequalities which are valid in two dimensions, namely, the Ladyzhenskaya inequality,
\be\label{ineqLadyzhenskaya}
 \|\bu \|_{L^4} \leq c_L |\bu|_{L^2}^{1/2} \|\bu\|_{H^1}^{1/2}, \quad \forall \bu \in V,
\ee
and the Br\'ezis-Gallouet inequality \cite{BrezisGallouet1980,FoiasManleyTemamTreve1983},
\be\label{ineqBrezisGallouet}
\|\bu\|_{L^\infty} \leq c_B \|\bu\|_{H^1} \left( 1 + \log\left( \frac{|A\bu|_{L^2}}{\lambda_1^{1/2} \|\bu\|_{H^1}} \right) \right)^{1/2}, \quad \forall \bu \in D(A).
\ee
Here, $c_L$ and $c_B$ denote nondimensional (scale invariant) constants, and $\|\cdot\|_{L^4}$ and $\|\cdot\|_{L^\infty}$ denote the usual norms in the Lebesgue spaces $L^4(\Omega)^2$ and $L^\infty(\Omega)^2$, respectively.

Property \eqref{PropBilTerm} and the Ladyzhenskaya inequality, \eqref{ineqLadyzhenskaya}, imply the following inequality for the bilinear term:
\be\label{estnonlineartermInnerProd}
 |\langle B(\bu_1,\bu_2),\bu_3\rangle_{V',V}| \leq c_L^2 |\bu_1|^{1/2}_{L^2}\|\bu_1\|^{1/2}_{H^1}|\bu_2|^{1/2}_{L^2}\|\bu_2\|^{1/2}_{H^1} \|\bu_3\|_{H^1}, \, \forall \bu_1,\bu_2, \bu_3 \in V,
\ee
from which it follows that 
\be\label{estnonlineartermV'}
 \|B(\bu_1,\bu_2)\|_{V'} \leq c_L^2 |\bu_1|^{1/2}_{L^2}\|\bu_1\|^{1/2}_{H^1}|\bu_2|^{1/2}_{L^2}\|\bu_2\|^{1/2}_{H^1}, \quad \forall \bu_1,\bu_2 \in V.
\ee

Also,
\begin{multline}\label{estnonlineartermL4L4L2}
 |(B(\bu_1,\bu_2),\bu_3)_{L^2}| \leq c_L^2 |\bu_1|_{L^2}^{1/2} \|\bu_1\|_{H^1}^{1/2} \|\bu_2\|_{H^1}^{1/2} |A\bu_2|_{L^2}^{1/2} |\bu_3|_{L^2}, \\
 \forall \bu_1 \in V, \forall \bu_2 \in D(A), \forall \bu_3 \in H,
\end{multline}
which implies that
\be\label{estnonlineartermHL4L4L2}
 |B(\bu_1,\bu_2)|_{L^2} \leq c_L^2 |\bu_1|_{L^2}^{1/2} \|\bu_1\|_{H^1}^{1/2} \|\bu_2\|_{H^1}^{1/2} |A\bu_2|_{L^2}^{1/2}, \quad \forall \bu_1 \in V, \forall \bu_2 \in D(A).
\ee

Moreover, from Br\'ezis-Gallouet inequality, \eqref{ineqBrezisGallouet}, it follows that 
\begin{multline}\label{estnonlineartermH}
 |(B(\bu_1,\bu_2),\bu_3)_{L^2}| \leq c_B \|\bu_1\|_{H^1} \left( 1 + \log\left( \frac{|A\bu_1|_{L^2}}{\lambda_1^{1/2} \|\bu_1\|_{H^1}} \right) \right)^{1/2} \|\bu_2\|_{H^1} |\bu_3|_{L^2}, \\
 \forall \bu_1 \in D(A), \forall \bu_2 \in V, \forall \bu_3 \in H,
\end{multline}
so that
\begin{multline}\label{estnonlineartermHfunc}
  |B(\bu_1,\bu_2)|_{L^2} \leq c_B \|\bu_1\|_{H^1} \|\bu_2\|_{H^1} \left( 1 + \log\left( \frac{|A\bu_1|_{L^2}}{\lambda_1^{1/2} \|\bu_1\|_{H^1}} \right) \right)^{1/2},\\  \forall \bu_1 \in D(A), \forall \bu_2 \in V.
\end{multline}

Moreover, we recall the following logarithmic inequalities for the bilinear term which were proved in \cite{Titi1987}:
\begin{multline}\label{ineqTiti1}
|(B(\bu_1,\bu_2),A\bu_3)_{L^2}| \leq c_T \|\bu_1\|_{H^1} \|\bu_2\|_{H^1} |A \bu_3|_{L^2} \left( 1 + \log\left( \frac{|A \bu_2|_{L^2}}{\lambda_1^{1/2}\|\bu_2\|_{H^1}}\right) \right)^{1/2}, \\
\forall \bu_1 \in V, \forall \bu_2, \bu_3 \in D(A),
\end{multline}
\begin{multline}\label{ineqTiti2}
|(B(\bu_1,\bu_2),A\bu_3)_{L^2}| \leq c_T \|\bu_1\|_{H^1} \|\bu_2\|_{H^1} |A \bu_3|_{L^2} \left( 1 + \log\left( \frac{|A \bu_1|_{L^2}}{\lambda_1^{1/2}\|\bu_1\|_{H^1}}\right) \right)^{1/2}, \\ \forall \bu_1, \bu_3 \in D(A), \forall \bu_2 \in V,
\end{multline}
where $c_T$ is a nondimensional constant.
%Throughout this work, we will always consider $I \subset \mathbb{R}$ as an interval closed and bounded on the left with left endpoint denoted by $t_0$, which represents the initial time.

It is well-known that given an initial data $\bu_0 \in H$ there exists a unique weak solution of \eqref{functionaleq} defined on $[t_0,\infty)$ and satisfying $\bu(t_0) = \bu_0$. Moreover, if $\bu_0 \in V$, then it is also true that there exists a unique strong solution of \eqref{functionaleq} defined on $[t_0,\infty)$ satisfying this initial condition. For completeness, we state these results below, where the notions of weak and strong solutions are also made more precise. 

%We remark that, although here we consider a time-independent forcing term $\bg$, the following statements are also true for $\bg \in L^{\infty}(t_0,\infty;H)$ and also, with a slight difference in some of the function spaces (which must carry the subscript ``loc''), to the more general case $\bg \in L^2_{\textnormal{loc}}(t_0,\infty;H)$

\begin{thm}[Existence and uniqueness of weak solutions]\label{thmexistuniqweaksolNSE}
 Let $\bu_0 \in H$. Then, there exists a unique (weak) solution $\bu$ of \eqref{functionaleq} satisfying $\bu(t_0) = \bu_0$ and 
 \[
  \bu \in  C([t_0,\infty);H) \cap L^2_{\textnormal{loc}}(t_0,\infty;V), \,\, \frac{\rd \bu}{\rd t} \in L^2_{\textnormal{loc}}(t_0,\infty;V').
 \]
 Moreover, $\bu$ depends continuously on the initial data $\bu_0$.
\end{thm}

\begin{thm}[Existence and uniqueness of strong solutions]\label{thmexistuniqstrongsolNSE}
 Let $\bu_0 \in V$. Then, there exists a unique (strong) solution $\bu$ of \eqref{functionaleq} satisfying $\bu(t_0) = \bu_0$ and 
 \[
  \bu \in  C([t_0,\infty);V) \cap L^2_{\textnormal{loc}}(t_0,\infty;D(A)), \,\, \frac{\rd \bu}{\rd t} \in L^2_{\textnormal{loc}}(t_0,\infty;H).
 \]
 Moreover, $\bu$ depends continuously on the initial data $\bu_0$.
\end{thm}

%Since our main intererest here is in the long-time behavior of solutions, from now on we assume that $I = [t_0, \infty)$.

Theorem \ref{thmexistuniqweaksolNSE} implies that, for each $t \in I = [t_0, \infty)$, we have a well-defined operator $S(t): H \rightarrow H$, given by
\be\label{defsemigroup}
 S(t)\bu_0 = \bu(t), \quad \forall \bu_0 \in H,
\ee
where $\bu(t)$ is the value at time $t$ of the unique (weak) solution of \eqref{NSEeqs} satisfying $\bu(t_0) = \bu_0$. The family of operators $\{S(t)\}_{t\geq t_0}$ is a semigroup for system \eqref{NSEeqs}.

We say that $\mB \subset H$ is a bounded \textit{absorbing} set if it is a bounded set with the property that, for every bounded subset $B \subset H$, there exists a time $T = T(B)$ for which $S(t)B \subset \mB$, for all $t \geq T$. The existence of a bounded absorbing set with respect to $\{S(t)\}_{t \geq t_0}$ was first obtained in \cite{FoiasProdi1967}.

Given a bounded absorbing set $\mB \subset H$, the global attractor $\mA$ of \eqref{functionaleq} is defined as the $\omega$-limit set of $\mB$ or, equivalently,
\[
 \mA = \bigcap_{t \geq t_0} S(t) \mB.
\]

The global attractor is a compact and invariant subset of $H$, which means that $S(t)\mA \subset \mA$, for all $t \geq t_0$. In other words, given $\bu_0 \in \mA$, the unique solution $\bu$ of \eqref{functionaleq} defined on $[t_0, \infty)$ and satisfying $\bu(t_0) = \bu_0$, remains in the global attractor for all later time, i.e. $\bu(t) \in \mA$, for all $t \geq t_0$. Moreover, $\mA$ is also a bounded subset of $V$, which implies, by Theorem \ref{thmexistuniqstrongsolNSE} and the invariance of $\mA$, that any trajectory in $\mA$ is a strong solution.

We recall the definition of the Grashof number, which is the nondimensional quantity given by
\[  G = \frac{|\bg|_{L^2}}{\nu^2\lambda_1}.
\]

In the next proposition we recall some uniform bounds of the attractor with respect to the $H$ and $V$ norms. The bounds are given in terms of the Grashof number $G$. For the proofs, we refer to any of the references listed above (\cite{bookcf1988,FMRT2001,Temambook2001}).

% The next proposition provides some uniform bounds which are valid for trajectories lying in the global attractor $\mA$. The bounds are given in terms of the Grashof number $G$. The proofs can be found in \cite{bookcf1988,Robinson2001,Temambook1997}. For the proof of estimate \eqref{boundAuperiodic}, we refer to \cite{DascFoiasJolly2010}, and for estimate \eqref{boundAunoslip}, we refer to \cite{FarhatLunasinTiti2015}.

In the statement below and in the remainder of this work, we denote by $c$ a generic absolute constant, whose value may change from line to line.

\begin{prop}\label{propunifbounds}
 For every $\bu \in \mA$, the following hold:
 \begin{enumerate}[(i)]
  \item In the case of periodic boundary conditions,
  \be |\bu|_{L^2} \leq \nu G, \quad \|\bu\|_{H^1} \leq \nu \lambda_1^{1/2} G,
  \ee
  %\be\label{boundAuperiodic} |A\bu(t)|_{L^2} \leq c \nu \lambda_1(1+G)^2.
  %\ee
  \item In the case of no-slip boundary conditions,
  \be |\bu|_{L^2} \leq \nu G, \quad \|\bu\|_{H^1} \leq c \nu \lambda_1^{1/2} G \Exp^{\frac{G^4}{2}},
  \ee
  %\be\label{boundAunoslip} |A\bu(t)|_{L^2} \leq c \nu \lambda_1 G [1 + (1+ G^2 \Exp^{G^4})(1 + \Exp^{G^4} + G^4 \Exp^{G^4})]^{1/2}.
  %\ee
 \end{enumerate}
\end{prop}

In order to simplify the notation, we  will represent the uniform bounds from Proposition \ref{propunifbounds} by writing
\be\label{unifboundsattractor}
%|\bu(t)|_{L^2} \leq \nu R_H, \quad \|\bu(t)\|_{H^1} \leq \nu \lambda_1^{1/2} R_V, \quad |A\bu(t)|_{L^2} \leq \nu\lambda_1 R_A, \quad \forall t \geq t_0,
|\bu|_{L^2} \leq M_0, \quad \|\bu\|_{H^1} \leq M_1, \quad \forall \bu \in \mA,
\ee
where $M_0$ and $M_1$ are dimensional constants depending on the Grashof number $G$, and whose values vary according to the boundary condition being considered, periodic or no-slip. Note, however, that $M_0 = \nu G$ in both cases.

\section{Discrete Data Assimilation Algorithm}\label{secDiscDataAssAlg}

As it was already pointed out in the Introduction, the purpose of our work is to consider discrete measurements in space and time, which may also be contaminated by errors, and to construct a data assimilation algorithm in order to recover asymptotically in time the exact reference solution, satisfying the 2D Navier-Stokes equations, corresponding to these measurements.

Since our goal here is to analyze the long-time behavior of solutions, in all the statements below we make the assumption that the reference solution $\bu$ is a trajectory in $\mA$, the global attractor of the 2D Navier-Stokes equations, recalled in section \ref{sec2DNSESetting}. We remark, however, that the same results still hold by assuming that $\bu$ is a solution of the 2D Navier-Stokes equations starting at a point $\bu(t_0) = \bu_0 ]in H$ with $t_0$ large enough so that the uniform bounds \eqref{unifboundsattractor} are also valid for $\bu$, up to a multiplicative absolute constant.

We know describe the necessary setup for introducing our algorithm. 

Let us denote by $\{t_n\}_{n\in\mathbb{N}}$ the sequence of instants of time in $[t_0, \infty)$ at which measurements are taken. We assume that 
\[
 t_n < t_{n+1}, \quad \forall n \in \mathbb{N},
\]
and 
\[
 t_n \to \infty, \quad \mbox{ as } n \to \infty.
\]
Moreover, we denote the maximum step size between successive measurements by a positive constant $\kappa$, so that
\[
 |t_{n+1} - t_n| \leq \kappa, \quad \forall n \in \mathbb{N}.
\]
Also, we assume that the data is collected from a spatial coarse mesh with resolution of size $h$.

In order to be able to use these discrete measurements in our algorithm, we must perform an interpolation in time and space. For the spatial interpolation, we consider a linear operator $I_h: L^2(\Omega)^2 \rightarrow L^2(\Omega)^2$, which is assumed to satisfy the following properties:

\renewcommand{\theenumi}{{P}\arabic{enumi}}
\begin{enumerate}
 \item\label{propintP1} There exists a positive constant $c_0$ such that 
 \be\label{eqpropintP1}
 |\varphi - I_h(\varphi)|_{L^2} \leq c_0 h \|\varphi\|_{H^1}, \quad \forall \varphi \in H^1(\Omega)^2.
\ee
 \item\label{propintP2} $I_h: L^2(\Omega)^2 \to L^2(\Omega)^2$ is a bounded operator, i.e. there exists a positive constant $c_1$ such that 
 \be\label{eqpropintP2}
  |I_h(\varphi)|_{L^2} \leq c_1 |\varphi|_{L^2}, \quad \forall \varphi \in L^2(\Omega)^2.
 \ee 
\end{enumerate}

An example of interpolant operator satisfying properties \eqref{propintP1} and \eqref{propintP2} is given by the low Fourier modes projector $P_m$, with $\lambda_1 m \leq 1/h^2$, where we recall that $\lambda_1$ is the first eigenvalue of the Stokes operator. Another more physical example is given by the sum of local spatial averages over finite volume elements (see, e.g., \cite{FoiasTiti1991,JonesTiti1992,JonesTiti1993}). 

Finally, we denote by $\eta_n$ the error associated to the measurements at time $t_n$, obtained after interpolation in space (for details on how to obtain $\eta_n$ in a particular example, see the Appendix). We assume that $\eta_n \in L^2(\Omega)^2$.

The observational measurements at each time $t_n$ are thus represented by
\be\label{repmeasurements}
 \tilde{\bu}(t_n) = I_h(\bu(t_n)) + \eta_n,
\ee
where $\bu$ is the unknown reference solution, satisfying \eqref{functionaleq}. 

Then, inspired by the continuous data assimilation algorithm introduced in \cite{Azouanietal2014}, we introduce the following discrete data assimilation algorithm for finding an approximate solution $\bv$ of the unknown reference solution $\bu$ of the two-dimensional Navier-Stokes equations. Given an arbitrary initial data $\bv_0$, we look for a function $\bv$ satisfying $\bv(t_0) = \bv_0$, the same boundary conditions for $\bu$, and the following system
\be\label{DiscDataAssProb1}
\frac{\partial \bv}{\partial t} - \nu \Delta \bv + (\bv \cdot \nabla) \bv + \nabla q = \f  - \beta \sum_{n=0}^\infty (I_h(\bv(t_n)) - \tilde{\bu}(t_n))) \chi_n,
\ee
\be\label{DiscDataAssProb2}
 \nabla \cdot \bv = 0,
\ee
where $\nu$ and $\f$ are the same kinematic viscosity parameter and forcing term from \eqref{NSEeqs}, respectively; $q$ is a modified pressure; $\tilde{\bu}(t_n)$ represents the observational measurements at time $t_n$, given in \eqref{repmeasurements}; $\chi_n$ is the characteristic function of the interval $[t_n,t_{n+1})$; and $\beta > 0$ is a relaxation (nudging) parameter. The purpose of $\beta$ is to force the coarse spatial scales of $\bv$ toward those of the reference solution $\bu$. 

Notice that, using the definition of $\tilde{\bu}(t_n)$ given in \eqref{repmeasurements} and the functional setting from section \ref{sec2DNSESetting}, we can rewrite system \eqref{DiscDataAssProb1}-\eqref{DiscDataAssProb2} in the following equivalent form
\be\label{DataAssEqDiscTimeGeneral}
\frac{\rd \bv}{\rd t} + \nu A\bv + B(\bv,\bv) = \bg  - \beta \sum_{n=0}^\infty P_{\sigma} (I_h(\bv(t_n)) - I_h(\bu(t_n))) \chi_n  + \beta \sum_{n=0}^{\infty} P_{\sigma} \eta_n \chi_n,
\ee
where $\bg = P_{\sigma}\f$, as before.

The existence and uniqueness of weak and strong solutions for the initial-value problem associated to \eqref{DataAssEqDiscTimeGeneral} can be proved by using the classical corresponding results already known for the 2D Navier-Stokes equations (Theorems \ref{thmexistuniqweaksolNSE} and \ref{thmexistuniqstrongsolNSE}). These are proved in the following two theorems. We remark that, as in Theorems \ref{thmexistuniqweaksolNSE} and \ref{thmexistuniqstrongsolNSE}, the result below is also valid in the more general case when $\bg \in L^{\infty}(t_0,\infty;H)$. 

\begin{thm}\label{thmexistuniqweaksolDiscDataAss}
Let $\bv_0 \in H$ and let $\bu$ be a trajectory in $\mA$. Then, there exists a unique (weak) solution $\bv$ of \eqref{DataAssEqDiscTimeGeneral} on $[t_0, \infty)$, satisfying $\bv(t_0) = \bv_0$ and
 \be\label{propvweaksol}
   \bv \in C([t_0,\infty);H) \cap L^2_{\textnormal{loc}} (t_0,\infty;V),\,\, \frac{\rd \bv}{\rd t} \in L^2_{\textnormal{loc}} (t_0,\infty;V').
 \ee
\end{thm}
\begin{proof}
Consider $\bh_0 = \bg - \beta(I_h(\bv_0) - \tilde{\bu}(t_0))$. Since $\bv_0 \in H$ and $\bh_0 \in H$, by Theorem \ref{thmexistuniqweaksolNSE}, there exists a unique weak solution $\bv^0$ of the 2D Navier-Stokes equations on $[t_0, \infty)$ corresponding to the forcing term $\bh_0$ and such that $\bv^0(t_0) = \bv_0$. 

Then, considering $\bh_1 = \bg - \beta(I_h(\bv^0(t_1) - \tilde{\bu}(t_1)) \in H$ and applying Theorem \ref{thmexistuniqweaksolNSE} once again, we have a unique weak solution $\bv^1$ of the 2D Navier-Stokes equations on $[t_1, \infty)$ corresponding to the forcing term $\bh_1$ and such that $\bv^1(t_1) = \bv^0(t_1) \in H$. 

Proceeding inductively, we have that for each $n \in \mathbb{N}$ there exists a unique weak solution $\bv^n$ of the 2D Navier-Stokes equations on $[t_n, \infty)$ corresponding to the forcing term $\bh_n = \bg - \beta(I_h(\bv^{n-1}(t_n) - \tilde{\bu}(t_n)) \in H$ and such that $\bv^n(t_n) = \bv^{n-1}(t_n) \in H$.

Let $\bv$ be the function defined on $[t_0, \infty)$ and given by
\[
 \bv(t) = \bv^n(t), \quad \forall t \in [t_n, t_{n+1}), \quad \forall n \in \mathbb{N}.
\]
By construction, $\bv$ is a solution of \eqref{DataAssEqDiscTimeGeneral} satisfying $\bv(t_0) = \bv_0$ and the properties in \eqref{propvweaksol}. Indeed, the equality $\bv^n(t_n) = \bv^{n-1}(t_n)$, valid for all $n \in \mathbb{N}$, guarantees that $\bv$ is a continuous function on $[t_0,\infty)$ in $H$. Moreover, since for every $n \in \mathbb{N}$ we have
\[
  \bv^n \in L^2_{\textnormal{loc}} (t_n,\infty;V),\,\, \frac{\rd \bv^n}{\rd t} \in L^2_{\textnormal{loc}} (t_n,\infty;V'),
\]
and the sequence $\{t_n\}_{n\in\mathbb{N}}$ of concatenating points is a countable set, it follows that
\[
 \bv \in L^2_{\textnormal{loc}} (t_0,\infty;V),\,\, \frac{\rd \bv}{\rd t} \in L^2_{\textnormal{loc}} (t_0,\infty;V').
\]

For the uniqueness, suppose that $\tilde{\bv}$ is another weak solution of \eqref{DataAssEqDiscTimeGeneral} on $[t_0, \infty)$ satisfying $\tilde{\bv}(t_0) = \bv_0$. Thus, $\tilde{\bv}|_{[t_0,t_1)}$ is a weak solution of the 2D Navier-Stokes equations on $[t_0,t_1)$ corresponding to the forcing term $\bh_0$ satisfying $\tilde{\bv}|_{[t_0,t_1)}(t_0) = \bv_0 = \bv|_{[t_0,t_1)}(t_0)$, so that $\tilde{\bv}|_{[t_0,t_1)} = \bv|_{[t_0,t_1)}$. But since $\bv, \tilde{\bv}\in \mC([t_0, \infty); H)$, they must also coincide on the closed interval $[t_0,t_1]$. Then, we can apply the same argument to the following interval $[t_1,t_2)$ and, proceeding inductively, to all subsequent intervals $[t_n, t_{n+1})$, $n \geq 2$. Therefore, $\bv= \tilde{\bv}$.
\end{proof}

\begin{thm}
Let $\bv_0 \in V$ and let $\bu$ be a trajectory in $\mA$. Then, there exists a unique (strong) solution $\bv$ of \eqref{DataAssEqDiscTimeGeneral} satisfying $\bv(t_0) = \bv_0$ and
 \be
   \bv \in C([t_0,\infty);V) \cap L^2_{\textnormal{loc}} (t_0,\infty;D(A)) ,\,\, \frac{\rd \bv}{\rd t} \in L^2_{\textnormal{loc}} (t_0,\infty;H).
 \ee
\end{thm}
\begin{proof}
The proof follows by analogous arguments of those in the proof of Theorem \ref{thmexistuniqweaksolDiscDataAss}, but using Theorem \ref{thmexistuniqstrongsolNSE} instead.
\end{proof}

\section{Asymptotic in time analysis for individual solutions}\label{secAsympAnalysis}

In this section, we analyze the asymptotic behavior in time of the difference between the solution $\bv$ of the discrete data assimilation algorithm \eqref{DataAssEqDiscTimeGeneral} for an arbitrary initial data $\bv_0$ and the unknown reference solution $\bu$ of the 2D Navier-Stokes equations corresponding to the discrete measurements. The results show that convergence of $\bv$ to $\bu$ follows up to a term depending on the errors associated to the observational measurements, provided $\beta$ is large enough and $\kappa$ and $h$ are sufficiently small.

First, in subsection \ref{subsecGalerkinProj}, we consider the particular case in which the interpolant operator $I_h$ is given by the low Fourier modes projector $P_m$, $m \in \mathbb{N}$. In this case, the initial data $\bv_0$ is allowed to be any element in $H$, and an asymptotic estimate of the difference between $\bv$ and $\bu$ is obtained with respect to the norm in $H$. Next, in subsection \ref{subsecGeneralInt}, we consider the more general case of an interpolant operator satisfying properties \eqref{propintP1} and \eqref{propintP2}, where we must consider $\bv_0 \in V$ and the corresponding asymptotic estimate is obtained with respect to the norm in $V$. The reason for this difference between each case will be explained within the next subsections.

\subsection{The case of projection on low Fourier modes}\label{subsecGalerkinProj}

In the particular case when the interpolant operator $I_h$ is given by the low Fourier modes projector $P_m$, $m \in \mathbb{N}$, the observational measurements at time $t_n$ are given by
\[ \tilde{\bu}(t_n) = P_m \bu(t_n) + \eta_n,
\]
where we recall that $\bu$ is the unknown reference solution of \eqref{functionaleq} and $\eta_n \in L^2(\Omega)^2$ is the error associated to the measurements at time $t_n$. Notice that $P_m\bu(t_n)$ is not known and all that is given is $\tilde{\bu}(t_n)$. 

The discrete data assimilation algorithm \eqref{DataAssEqDiscTimeGeneral} is therefore given in this case as
\be\label{DataAssEqDiscTimeFourier}
  \frac{\rd \bv}{\rd t} + \nu A\bv + B(\bv,\bv) = \bg - \beta \sum_{n=0}^\infty (P_m\bv(t_n) - P_m\bu(t_n)) \chi_n + \beta \sum_{n=0}^{\infty} P_{\sigma} \eta_n \chi_n.
\ee

In the following theorem, we analyze the asymptotic behavior in time of the solution $\bv$ of \eqref{DataAssEqDiscTimeFourier} corresponding to an initial data $\bv_0 \in H$. We assume that the sequence of errors $\{\eta_n\}_{n\in\mathbb{N}}$ is bounded in $L^2(\Omega)^2$, with a bound given by a positive constant $E_0$, which represents the maximum ``size'' of the errors. In applications, the constant $E_0$ would be given in terms of the accuracy associated to the experimental devices used for obtaining the measurements (see the Appendix for an explicit estimate of this type in the case of an interpolant operator given by local averages over finite volume elements).

Moreover, we assume suitable conditions on $\beta$, $\kappa$ and $m$, which are expressed in terms of the uniform bounds of $\bu$, given by the constants $M_0$ and $M_1$ from \eqref{unifboundsattractor}.

We recall that $c$ denotes a generic absolute constant.

\begin{thm}\label{thmasympestwFourier}
 Let $\bu$ be a trajectory in $\mA$. Consider $\bv_0 \in B_H(M_0)$, and let $\bv$ be the unique solution of \eqref{DataAssEqDiscTimeFourier} on $[t_0,\infty)$ satisfying $\bv(t_0) = \bv_0$. Assume that $\{\eta_n\}_{n\in\mathbb{N}}$ is a bounded sequence in $L^2(\Omega)^2$, namely, there exists a constant $E_0 \geq 0$ such that
 \be\label{bounderrorH}
  |\eta_n|_{L^2} \leq E_0, \quad \forall n \in \mathbb{N}.
 \ee
 If $\beta$ and $m$ are large enough such that
\be\label{condbetamodes}
 \beta \geq c \frac{M_1^2}{\nu},
\ee
\be\label{mCond}
 \lambda_{m+1} \geq 6\frac{\beta}{\nu},
\ee
and $\kappa$ is small enough such that
\begin{multline}\label{condkappamodes}
 \kappa \leq \frac{c}{\beta} \min\left\{ 1, \frac{\nu}{M_0 + E_0}, \frac{\nu^2}{ (M_0 + E_0)^2}, \frac{\nu^{3/2}\beta^{1/2}}{ M_0 M_1},  \frac{\nu^2 \lambda_1^{1/2}}{M_0 M_1}, \frac{(\nu \lambda_1)^{1/3}}{\beta^{1/3}}, \right.\\
 \left. \frac{(\nu \lambda_1)^{1/2}}{\beta^{1/2}}, \frac{(\nu \lambda_1)^2}{\beta^2}\right\}
\end{multline}
then
\[
 \limsup_{t\rightarrow \infty} |\bv(t) - \bu(t)|_{L^2} \leq c E_0.
\]
Moreover, if $E_0 = 0$, then $\bv(t) \rightarrow \bu(t)$ in $H$, exponentially, as $t \rightarrow \infty$.
\end{thm}
\begin{proof}
Denote $\bw = \bv -\bu$. Subtracting \eqref{functionaleq} from \eqref{DataAssEqDiscTimeFourier}, we obtain the following functional equation for $\bw$:
 \be\label{wEqDiscDataAssFourier}
  \frac{\rd \bw}{\rd t} + \nu A\bw + B(\bw,\bu) + B(\bu,\bw) + B(\bw,\bw) = -\beta \sum_{n=0}^\infty P_m \bw(t_n)\chi_n + \beta\sum_{n=0}^\infty P_{\sigma} \eta_n \chi_n,
 \ee
which holds in $L^2_{\textnormal{loc}}(t_0,\infty;V')$.

When applied to $\bw$, \eqref{wEqDiscDataAssFourier} yields
 \begin{multline}\label{energyeqwFourier}
  \frac{1}{2} \frac{\rd}{\rd t} |\bw|^2_{L^2} + \nu \|\bw\|^2_{H^1} = \\
= - \langle B(\bw,\bu),\bw\rangle_{V',V} -\beta \sum_{n=0}^\infty (P_m(\bw(t_n) - \bw),\bw)_{L^2}\chi_n + \beta \sum_{n=0}^\infty (\eta_n, \bw)_{L^2} \chi_n \\
= - \langle B(\bw,\bu),\bw\rangle_{V',V}-\beta |P_m \bw|_{L^2}^2 -\beta\sum_{n=0}^\infty (P_m(\bw(t_n) - \bw), \bw)_{L^2} \chi_n + \\
+ \beta \sum_{n=0}^\infty (\eta_n, \bw)_{L^2} \chi_n,
 \end{multline}
where we used property \eqref{PropBilTerm} of the bilinear term.

Now we estimate each term on the right-hand side of \eqref{energyeqwFourier}. Since $\bu$ is a trajectory in $\mA$, we can use the bounds from \eqref{unifboundsattractor}.

Using \eqref{PropBilTerm}, \eqref{estnonlineartermInnerProd} and Young's inequality, we obtain that 
 \begin{eqnarray}\label{estA}
  |\langle B(\bw,\bu),\bw\rangle_{V',V}| &\leq& c \|\bu\|_{H^1} |\bw|_{L^2} \|\bw\|_{H^1} \nonumber\\
                           &\leq& \frac{\nu}{6} \|\bw\|_{H^1}^2 + \frac{c}{\nu} \|\bu\|^2_{H^1} |\bw|^2_{L^2} \nonumber\\
&\leq& \frac{\nu}{6} \|\bw\|_{H^1}^2 + c \frac{M_1^2}{\nu} |\bw|^2_{L^2}.
 \end{eqnarray}
 
Notice that 
 \begin{eqnarray}\label{estB}
  -\beta|P_m \bw|^2_{L^2} &=& -\beta|\bw|^2_{L^2} + \beta | Q_m \bw |^2_{L^2} \nonumber\\
                            &\leq& -\beta |\bw|^2_{L^2} + \frac{\beta}{\lambda_{m+1}}  \| Q_m \bw \|^2_{H^1} \nonumber\\
                            &\leq& -\beta |\bw|^2_{L^2} + \frac{\beta}{\lambda_{m+1}}  \| \bw \|^2_{H^1} \nonumber\\
                            &\leq& -\beta |\bw|^2_{L^2} + \frac{\nu}{6} \| \bw \|^2_{H^1},
 \end{eqnarray}
where in the last inequality we used hypothesis \eqref{mCond}.

Also, using the bound from hypothesis \eqref{bounderrorH}, we have
\be
  \beta |(\eta_n, \bw)_{L^2}| \leq \beta |\eta_n|_{L^2} |\bw|_{L^2} \leq \frac{\beta}{2} E_0^2 + \frac{\beta}{2}|\bw|_{L^2}^2.
\ee

Moreover, 
 \begin{multline}\label{estC}
  \beta |(P_m(\bw(t_n) - \bw(t)),\bw(t))_{L^2}| = \beta |(\bw(t_n) - \bw(t),P_m\bw(t))_{L^2}| \\
  = \beta \left| \left( \int_{t_n}^t \frac{\rd \bw}{\rd s}(s) \rd s, P_m \bw(t) \right)_{L^2}\right| \\
  \leq \beta\left( \int_{t_n}^t \left\| \frac{\rd \bw}{\rd s}(s)\right\|_{V'} \rd s \right) \|\bw(t)\|_{H^1} \\
  \leq \frac{\nu}{6} \|\bw(t)\|_{H^1}^2 + \frac{c\beta^2}{\nu} \left( \int_{t_n}^t \left\| \frac{\rd \bw}{\rd s}(s)\right\|_{V'} \rd s\right)^2.
 \end{multline}

%  Now, using H\"older's inequality, we have
%  \be\label{estD}
%   \left( \int_{t_n}^t \left\| \frac{\rd \bw}{\rd s}(s)\right\|_{V'} \rd s\right)^2 \leq \kappa \int_{t_n}^t \left\| \frac{\rd \bw}{\rd s}(s)\right\|_{V'}^2 \rd s.
%  \ee
%  

Furthermore, from \eqref{wEqDiscDataAssFourier}, \eqref{estnonlineartermV'} and Poincar\'e inequality \eqref{ineqPoincare}, we obtain that
 \begin{multline}
  \left\|\frac{\rd \bw}{\rd s}(s) \right\|_{V'} \leq \nu \|\bw(s)\|_{H^1} + \| B(\bw(s), \bu(s))\|_{V'} + \| B(\bu(s), \bw(s))\|_{V'} + \\
                                                     + \|B(\bw(s), \bw(s))\|_{V'} + \beta \|\bw(t_n) - \bw(s)\|_{V'} + \beta \|\bw(s)\|_{V'} + \beta \|\eta_n\|_{V'}\\
                                                \leq \nu \|\bw(s)\|_{H^1} + c |\bu(s)|_{L^2}^{1/2} \|\bu(s)\|_{H^1}^{1/2} |\bw(s)|_{L^2}^{1/2} \|\bw(s)\|_{H^1}^{1/2} + c |\bw(s)|_{L^2} \|\bw(s)\|_{H^1} +\\
                                                  + \beta \int_{t_n}^s \left\| \frac{\rd \bw}{\rd \tau} (\tau) \right\|_{V'}\rd \tau + \frac{\beta}{\lambda_1^{1/2}} |\bw(s)|_{L^2} + \frac{\beta}{\lambda_1^{1/2}} E_0 \\
                                                \leq \nu \|\bw(s)\|_{H^1} + c (M_0 M_1)^{1/2} |\bw(s)|_{L^2}^{1/2}\|\bw(s)\|_{H^1}^{1/2} + c |\bw(s)|_{L^2} \|\bw(s)\|_{H^1} +\\
                                                    + \beta \int_{t_n}^s \left\| \frac{\rd \bw}{\rd \tau} (\tau) \right\|_{V'}\rd \tau + \frac{\beta}{\lambda_1^{1/2}} |\bw(s)|_{L^2} + \frac{\beta}{\lambda_1^{1/2}} E_0.
 \end{multline}
Integrating with respect to $s$ from $t_n$ to $t \in [t_n, t_{n+1})$, it follows that
 \begin{multline}\label{estE}
  \int_{t_n}^t \left\|\frac{\rd \bw}{\rd s}(s) \right\|_{V'} \rd s \leq 
  \int_{t_n}^t \left( \nu \|\bw(s)\|_{H^1} + c (M_0 M_1)^{1/2} |\bw(s)|_{L^2}^{1/2}\|\bw(s)\|_{H^1}^{1/2} + \right.\\ 
  \left. c |\bw(s)|_{L^2} \|\bw(s)\|_{H^1} + \frac{\beta}{\lambda_1^{1/2}} |\bw(s)|_{L^2} + \frac{\beta}{\lambda_1^{1/2}} E_0 \right) \rd s + \beta\kappa \int_{t_n}^t \left\| \frac{\rd \bw}{\rd s}(s) \right\|_{V'} \rd s,
 \end{multline}
where we used that
 \be\label{doubleinttime}
  \int_{t_n}^t \int_{t_n}^s \left\| \frac{\rd \bw}{\rd \tau} (\tau) \right\|_{V'}\rd \tau \rd s \leq \int_{t_n}^t \int_{t_n}^t \left\| \frac{\rd \bw}{\rd \tau} (\tau) \right\|_{V'} \rd \tau \rd s 
  \leq \kappa \int_{t_n}^t \left\| \frac{\rd \bw}{\rd \tau} (\tau) \right\|_{V'} \rd \tau.
 \ee
From condition \eqref{condkappamodes} on $\kappa$, with $c \leq 1/2$, we have in particular that $\beta \kappa \leq 1/2$. Thus, we obtain from \eqref{estE} that
 \begin{multline}\label{estF}
   \int_{t_n}^t \left\|\frac{\rd \bw}{\rd s}(s) \right\|_{V'} \rd s \leq 
  c \int_{t_n}^t \left( \nu \|\bw(s)\|_{H^1} + c (M_0 M_1)^{1/2} |\bw(s)|_{L^2}^{1/2}\|\bw(s)\|_{H^1}^{1/2} + \right.\\
  \left. |\bw(s)|_{L^2} \|\bw(s)\|_{H^1} + \frac{\beta}{\lambda_1^{1/2}} |\bw(s)|_{L^2} + \frac{\beta}{\lambda_1^{1/2}} E_0 \right) \rd s 
 \end{multline}
Now, from H\"older inequality, it follows that
 \be\label{estG}
  \left( \int_{t_n}^t \left\| \frac{\rd \bw}{\rd s}(s)\right\|_{V'} \rd s\right)^2 \leq c\kappa \int_{t_n}^t \varphi(s) \rd s + \frac{c \beta^2 \kappa^2}{\lambda_1} E_0^2,
 \ee
where
 \be\label{defvarphi}
  \varphi(s) = \nu^2 \|\bw(s)\|_{H^1}^2  + M_0 M_1 |\bw(s)|_{L^2} \|\bw(s)\|_{H^1} + |\bw(s)|_{L^2}^2 \|\bw(s)\|_{H^1}^2 + \frac{\beta^2}{\lambda_1} |\bw(s)|_{L^2}^2.  
 \ee
Thus, using estimates \eqref{estA}-\eqref{estC} and \eqref{estG} in \eqref{energyeqwFourier}, we obtain that
\begin{multline}\label{energyineqdiffform}
 \frac{\rd}{\rd t}|\bw(t)|^2_{L^2} + \nu \|\bw(t)\|_{H^1}^2 \leq -\left( \beta - c \frac{M_1^2}{\nu} \right) |\bw(t)|_{L^2}^2 + \frac{c\beta^2 \kappa}{\nu} \sum_{n=0}^\infty \chi_n(t) \int_{t_n}^t \varphi(s) \rd s +\\
 + \beta\left( 1 + \frac{c\beta^3 \kappa^2}{\nu \lambda_1} \right) E_0^2,
\end{multline}
for all $t \in [t_0,\infty)$.

Moreover, using condition \eqref{condbetamodes} on $\beta$ with a suitable absolute constant $c$, we have
\begin{multline}\label{energyineqdiffform1}
 \frac{\rd}{\rd t}|\bw(t)|^2_{L^2} + \nu \|\bw(t)\|_{H^1}^2 \leq -\frac{\beta}{2} |\bw(t)|_{L^2}^2 + \frac{c\beta^2 \kappa}{\nu} \sum_{n=0}^\infty \chi_n(t) \int_{t_n}^t \varphi(s) \rd s +\\
 + \beta\left( 1 + \frac{c\beta^3 \kappa^2}{\nu \lambda_1} \right) E_0^2,
\end{multline}
%Consider $R > 0$ satisfying 
% \be
%  M > \max\{ |\bw(t_0)|_{L^2}, 3\nu E \}.
% \ee
Denote
\be
 R = 2 (M_0 + E_0).
\ee

Since $\bw \in \mC([t_0,\infty);H)$, and
\[
 |\bw(t_0)|_{L^2} \leq |\bv(t_0)|_{L^2} + |\bu(t_0)|_{L^2} \leq 2 M_0 \leq R,
\]
there exists $\tau \in (t_0,\infty)$ such that 
 \be\label{boundwsmalltime}
  |\bw(t)|_{L^2} \leq 2R, \quad \forall t \in [t_0,\tau].
 \ee
Define 
 \be\label{deftildet}
  \tilde{t} = \sup \left\{ \tau \in [t_0,\infty) : \sup_{t\in [t_0,\tau]}|\bw(t)|_{L^2} \leq 2R \right\}.
 \ee
 
Suppose that $\tilde{t} < t_1$. Then, integrating \eqref{energyineqdiffform} from $t_0$ to $t \leq \tilde{t}$ and using an estimate similar to \eqref{doubleinttime}, we obtain that 
 \begin{multline}\label{energyineqwintform}
   |\bw(t)|_{L^2}^2 - |\bw(t_0)|_{L^2}^2 + \nu \int_{t_0}^t \|\bw(s)\|_{H^1}^2 \rd s \leq \\
\leq - \frac{\beta}{2} \int_{t_0}^t |\bw(s)|_{L^2}^2 \rd s + \frac{c\beta^2 \kappa^2}{\nu} \int_{t_0}^t \varphi(s) \rd s + \beta\kappa \left( 1 + \frac{c\beta^3 \kappa^2}{\nu \lambda_1} \right) E_0^2.
\end{multline}

Using Young's inequality to estimate the second term in the definition of $\varphi$ and the fact that $|\bw(t)|_{L^2} \leq R$, for all $t \in [t_0,\tilde{t}]$, we have
\begin{multline}
|\bw(t)|_{L^2}^2 - |\bw(t_0)|_{L^2}^2 + \nu \int_{t_0}^t \|\bw(s)\|_{H^1}^2 \rd s \leq - \frac{\beta}{2} \int_{t_0}^t |\bw(s)|_{L^2}^2 \rd s \\
+ \frac{c\beta^2 \kappa^2}{\nu} \int_{t_0}^t \left( \nu^2 \|\bw(s)\|_{H^1}^2 
 + \frac{(M_0 M_1)^2}{\nu^2} |\bw(s)|_{L^2}^2 + R^2 \|\bw(s)\|_{H^1}^2 + \frac{\beta^2}{\lambda_1}|\bw(s)|_{L^2}^2 \right)\rd s \\
+ \beta\kappa \left( 1 + \frac{c\beta^3 \kappa^2}{\nu \lambda_1} \right) E_0^2 .
\end{multline}
 
After some rearrangement,
 \begin{multline}\label{energyineqwintform2}
 |\bw(t)|_{L^2}^2 - |\bw(t_0)|_{L^2}^2 + \nu\left(1 - c\beta^2 \kappa^2\left( 1+\frac{R^2}{\nu^2} \right)\right) \int_{t_0}^t \|\bw(s)\|_{H^1}^2 \rd s \leq \\
 \leq -\left( \frac{\beta}{2} - c\frac{(\beta\kappa M_0 M_1)^2}{\nu^3} - \frac{c \beta^4 \kappa^2}{\nu \lambda_1}\right) \int_{t_0}^t |\bw(s)|_{L^2}^2 \rd s 
 + \beta\kappa \left( 1 + \frac{c\beta^3 \kappa^2}{\nu \lambda_1} \right) E_0^2. 
 \end{multline} 

Then, using the smallness condition \eqref{condkappamodes} on $\kappa$, with a suitable absolute constant $c$, it follows that
\be\label{energyineqwintform4}
|\bw(t)|_{L^2}^2 - |\bw(t_0)|_{L^2}^2 + \frac{\nu}{2} \int_{t_0}^t \|\bw(s)\|_{H^1}^2 \rd s \leq - \frac{\beta}{2} \int_{t_0}^t |\bw(s)|_{L^2}^2 \rd s + c E_0^2,
\ee
which implies in particular that 
\be\label{boundL2Vw}
 \int_{t_0}^t \|\bw(s)\|_{H^1}^2 \rd s \leq \frac{2}{\nu}|\bw(t_0)|_{L^2}^2 + \frac{c}{\nu} E_0^2, \quad \forall t \in [t_0,\tilde{t}].
\ee

Using Poincar\'e inequality \eqref{ineqPoincare} to estimate all the terms in the definition of $\varphi$, given in \eqref{defvarphi}, in terms of $\|\bw\|_{H^1}^2$, it follows from \eqref{energyineqdiffform1} that, for all $t \in [t_0, \tilde{t}]$,
\begin{multline}\label{energyineqdiffform2}
 \frac{\rd}{\rd t}|\bw|^2_{L^2}  \leq  -\frac{\beta}{2} |\bw|_{L^2}^2 + \frac{c\beta^2 \kappa}{\nu} \left( \nu^2 + \frac{M_0 M_1}{ \lambda_1^{1/2}} + R^2 + \frac{\beta^2}{\lambda_1^2}\right) \int_{t_0}^t \|\bw(s)\|_{H^1}^2 \rd s  +\\
  + \beta\left( 1 + \frac{c\beta^3 \kappa^2}{\nu \lambda_1} \right) E_0^2 .
\end{multline}

Estimating the right-hand side of \eqref{energyineqdiffform2} by using the bound from \eqref{boundL2Vw}, and then integrating from $t_0$ to $t \leq \tilde{t}$, it follows that
\be\label{ineqwGronwall}
 |\bw(t)|_{L^2}^2 \leq |\bw(t_0)|_{L^2}^2\Exp^{-\frac{\beta}{2}(t-t_0)} + ( \gamma|\bw(t_0)|_{L^2}^2 + \sigma E_0^2 + 2 E_0^2) ( 1 - \Exp^{-\frac{\beta}{2}(t-t_0)}), 
\ee
where
\be
 \gamma = c\beta\kappa \left( 1 + \frac{M_0 M_1}{\nu^2\lambda_1^{1/2}} + \frac{R^2}{\nu^2} + \frac{\beta^2}{(\nu\lambda_1)^2} \right)
\ee
and
\be
 \sigma = \gamma + \frac{c\beta^3 \kappa^2}{\nu\lambda_1}.
\ee

Since $|\bw(t_0)|_{L^2} \leq R$, we also have
\be\label{ineqwGronwall2}
 |\bw(t)|_{L^2}^2 \leq R^2\Exp^{-\frac{\beta}{2}(t-t_0)} + ( \gamma R^2 + \sigma E_0^2 + 2 E_0^2) ( 1 - \Exp^{-\frac{\beta}{2}(t-t_0)}).
\ee

Now, using condition \eqref{condkappamodes} on $\kappa$ with a suitable absolute constant $c$, we obtain $\gamma + \sigma \leq 1/2$, so that
\[
 \gamma R^2 + \sigma E_0^2 + 2 E_0^2 \leq \frac{R^2}{2} + 2 E_0^2 \leq R^2.
\]
Thus, it follows from \eqref{ineqwGronwall2} that
\[
 |\bw(t)|_{L^2} \leq R, \quad \forall t \in [t_0, \tilde{t}].
\]
In particular, $|\bw(\tilde{t})|_{L^2} \leq R$, and from the definition of $\tilde{t}$ in \eqref{deftildet} we conclude that $\tilde{t} \geq t_1$. Therefore, we also have $|\bw(t_1)|_{L^2} \leq R$ and we can apply the same previous arguments to obtain that $\tilde{t} \geq t_2$ and $|\bw(t_2)|_{L^2} \leq R$. Continuing inductively, we obtain that $\tilde{t} \geq t_n$, for all $n \geq 0$. Then, analogously to \eqref{ineqwGronwall}, one obtains that
\be\label{ineqwGronwall3}
 |\bw(t)|_{L^2}^2 \leq |\bw(t_n)|_{L^2}^2\Exp^{-\frac{\beta}{2}(t-t_n)} + \left( \gamma |\bw(t_n)|_{L^2}^2 + \sigma E_0^2 + 2 E_0^2\right) ( 1 - \Exp^{-\frac{\beta}{2}(t-t_n)}), 
\ee
for all $t \in [t_n,t_{n+1}]$ and for all $n \geq 0$. 

From \eqref{ineqwGronwall3}, it follows in particular that 
\[
 |\bw(t_{n+1})|_{L^2} \leq \theta |\bw(t_n)|_{L^2} + c E_0, \quad \forall n \geq 0,
\]
where
\[
 \theta = \left( \Exp^{-\frac{\beta\kappa}{2}} + \gamma ( 1 - \Exp^{-\frac{\beta\kappa}{2}})\right)^{1/2} < 1.
\]

Thus,
\be\label{ineqwGronwall4}
 |\bw(t_n)|_{L^2} \leq \theta^n|\bw(t_0)|_{L^2} + c E_0 \sum_{j = 0}^{n-1} \theta^j, \quad \forall n \geq 1.
\ee

From \eqref{ineqwGronwall3} and \eqref{ineqwGronwall4}, it follows that
\be\label{ineqwGronwall5}
 |\bw(t)|_{L^2} \leq \theta^n|\bw(t_0)|_{L^2} + c E_0 \left( 1 + \sum_{j = 0}^{n-1} \theta^j \right), \quad \forall t \in [t_n,t_{n+1}], \forall n \geq 1.
\ee

Therefore,
\[
 \limsup_{t\rightarrow \infty} |\bw(t)|_{L^2} \leq c E_0.
\]

Moreover, if $E_0 = 0$, we have from \eqref{ineqwGronwall5} that 
\[
 |\bw(t)|_{L^2} \leq \theta^n|\bw(t_0)|_{L^2},\quad \forall t \in [t_n,t_{n+1}], \forall n \geq 1,
\]
and thus $\bw(t)$ converges exponentially to $0$ in $H$ as $t \rightarrow \infty$.
\end{proof}

\subsection{The general interpolant case}\label{subsecGeneralInt}

We consider an interpolant operator $I_h: L^2(\Omega)^2 \rightarrow L^2(\Omega)^2$ satisfying properties \eqref{propintP1} and \eqref{propintP2}, given in \eqref{eqpropintP1} and \eqref{eqpropintP2}, respectively. The following theorem provides an asymptotic in time estimate of the difference between the unknown reference solution of \eqref{functionaleq} and the corresponding solution $\bv$ of equation \eqref{DataAssEqDiscTimeGeneral} satisfying an initial data $\bv_0 \in V$. 

Here, we assume that the sequence of errors $\{\eta_n\}_{n\in\mathbb{N}}$ is bounded in $H^1(\Omega)^2$. We remark that this is a natural assumption that can be verified in applications by using the given parameters associated to the model. As an example, we obtain in the Appendix an explicit bound in the $H^1(\Omega)^2$ norm of the sequence $\{\eta_n\}_{n\in\mathbb{N}}$ for the particular case of an interpolant operator given by local averages over finite volume elements. 

The main difference between the proof below and the proof of Theorem \ref{thmasympestwFourier} is that the general interpolant $I_h$ does not provide us with enough freedom for being able to work with the norm in $V'$ of the time derivative $\rd \bw/\rd t$, and we must use the norm in $L^2$ instead (compare \eqref{est5} with \eqref{estC}). Thus, in order for the calculations to make sense, $\bv$ has to be a strong solution of \eqref{DataAssEqDiscTimeGeneral}, which justifies the choice of an initial data $\bv_0$ in $V$.

\begin{thm}\label{thmasympestwGeneral}
 Let $\bu$ be a trajectory in $\mA$. Consider $\bv_0 \in B_V(M_1)$ and let $\bv$ be the unique solution of \eqref{DataAssEqDiscTimeGeneral} on $[t_0,\infty)$ satisfying $\bv(t_0) = \bv_0$. Assume that $\{\eta_n\}_{n\in\mathbb{N}}$ is a bounded sequence in $H^1(\Omega)^2$, namely, there exists a constant $E_1 \geq 0$ such that
 \be\label{bounderrorV}
  \|\eta_n\|_{H^1} \leq E_1, \quad \forall n \in \mathbb{N}.
 \ee
 If $\beta$ is large enough such that
\be\label{condbetag0}
 \beta \geq c \frac{(M_1 + E_1)^2}{\nu} \left( 1 + \log \left(\frac{M_1 + E_1}{\nu \lambda_1^{1/2}}\right)\right)
\ee
and $\kappa, h$ are small enough such that
\be\label{condkappag0}
 \kappa \leq \frac{c}{\beta} \min \left\{ 1, \frac{\nu^{3/2} \beta^{1/2}}{M_0 M_1},\frac{ \nu^2 \lambda_1^{1/2}}{ M_0 M_1}, \frac{\nu^2 \lambda_1}{(M_1 + E_1)^2}, \frac{(\nu \lambda_1)^{1/2}}{\beta^{1/2}}, \frac{(\nu \lambda_1)^2}{\beta^2} \right\},
\ee
\be\label{condh}
 h \leq \frac{1}{2c_0}\left( \frac{\nu}{\beta}\right)^{1/2},
\ee
then
\[
 \limsup_{t\rightarrow \infty} \|\bv(t) - \bu(t)\|_{H^1} \leq c E_1.
\]
Moreover, if $E_1 = 0$, then $\bv(t) \rightarrow \bu(t)$ in $V$ exponentially as $t \rightarrow \infty$.
\end{thm}
\begin{proof}
 Denote $\bw = \bv - \bu$. Subtracting \eqref{DataAssEqDiscTimeGeneral} from \eqref{functionaleq}, we obtain that
 \be\label{wEqDiscDataAssGeneral}
  \frac{\rd \bw}{\rd t} + \nu A\bw + B(\bw,\bu) + B(\bu,\bw) + B(\bw,\bw) = -\beta \sum_{n=0}^\infty P_{\sigma}I_h (\bw(t_n))\chi_n + \beta\sum_{n=0}^\infty P_{\sigma} \eta_n \chi_n,
 \ee
 which holds in $L^2_{\textnormal{loc}}(t_0,\infty;H)$.
 
Taking the inner product in $H$ of \eqref{wEqDiscDataAssGeneral} with $A\bw$, we have
 \begin{multline}\label{energyeqwGeneral}
  \frac{1}{2} \frac{\rd}{\rd t} \|\bw\|^2_{H^1} + \nu |A\bw|^2_{L^2} = \\
  = - ( B(\bw,\bu),A\bw )_{L^2} - ( B(\bu,\bw),A\bw )_{L^2} - ( B(\bw,\bw),A\bw )_{L^2}  - \\
 -\beta \sum_{n=0}^\infty (I_h(\bw(t_n)),A\bw)_{L^2}\chi_n + \beta \sum_{n=0}^\infty (\eta_n, A\bw)_{L^2} \chi_n \\
=  - ( B(\bw,\bu),A\bw )_{L^2} - ( B(\bu,\bw),A\bw )_{L^2} - ( B(\bw,\bw),A\bw )_{L^2}  - \\
-\beta \|\bw\|_{H^1}^2 -\beta (I_h(\bw) - \bw, A\bw)_{L^2} -\beta\sum_{n=0}^\infty (I_h(\bw(t_n) - \bw), A\bw)_{L^2} \chi_n + \\
+ \beta \sum_{n=0}^\infty (\eta_n, A\bw)_{L^2} \chi_n.
 \end{multline}
 
Next, we estimate each term on the right-hand side of \eqref{energyeqwGeneral}. We use the bounds for $\bu$ given in \eqref{unifboundsattractor}.

Using inequalities \eqref{ineqTiti1} and \eqref{ineqTiti2}, we obtain that
\be\label{est1}
 |( B(\bw,\bu),A\bw )_{L^2}| \leq c_T M_1 \|\bw\|_{H^1} |A\bw|_{L^2} \left( 1 + \log\left( \frac{|A\bw|_{L^2}}{\lambda_1^{1/2} \|\bw\|_{H^1}} \right) \right)^{1/2},
\ee
\begin{multline}\label{est2}
 |( B(\bu,\bw),A\bw )_{L^2}| \leq c_T M_1 \|\bw\|_{H^1} |A\bw|_{L^2} \left( 1 + \log\left( \frac{|A\bw|_{L^2}}{\lambda_1^{1/2} \|\bw\|_{H^1}} \right) \right)^{1/2}
\end{multline}
and
\be\label{est3}
 |( B(\bw,\bw),A\bw )_{L^2}| \leq c_T \|\bw\|_{H^1}^2 |A\bw|_{L^2} \left( 1 + \log\left( \frac{|A\bw|_{L^2}}{\lambda_1^{1/2} \|\bw\|_{H^1}} \right) \right)^{1/2}.
\ee

Using property \eqref{propintP1} of $I_h$, it follows that 
\begin{multline}\label{est3a}
 \beta |(I_h(\bw) - \bw, A\bw)_{L^2}| \leq c_0 \beta h\|\bw\|_{H^1} |A\bw|_{L^2} \\ \leq \frac{\beta}{4}\|\bw\|_{H^1}^2 + \beta c_0^2 h^2 |A\bw|_{L^2}^2 
 \leq \frac{\beta}{4}\|\bw\|_{H^1}^2 + \frac{\nu}{4}|A\bw|_{L^2}^2,
\end{multline}
where we used hypothesis \eqref{condh}.

Also, using hypothesis \eqref{bounderrorV}, we have that
\be\label{est4}
 \beta |(\eta_n, A\bw)_{L^2}| \leq \beta \|\eta_n\|_{H^1} \|\bw\|_{H^1} \leq \beta E_1^2 + \frac{\beta}{4}\|\bw\|_{H^1}^2.
\ee

Moreover, using property \eqref{propintP2} of $I_h$, we obtain that
\begin{multline}\label{est5}
 \beta |(I_h (\bw(t_n) - \bw(t)), A\bw(t))| \leq \beta |A\bw(t)|_{L^2} c_1 \int_{t_n}^t \left| \frac{\rd \bw}{\rd s}(s)\right|_{L^2} \rd s \\
 \leq \frac{\nu}{4}|A\bw(t)|_{L^2}^2 + \frac{c\beta^2}{\nu} \left( \int_{t_n}^t \left| \frac{\rd \bw}{\rd s}(s)\right|_{L^2} \rd s\right)^2
\end{multline}

Now, from \eqref{wEqDiscDataAssGeneral} and using property \eqref{propintP2} of $I_h$ and Poincar\'e inequality \eqref{ineqPoincare}, we obtain that, for all $t \in [t_n, t_{n+1})$,
\begin{multline}\label{est6}
 \left|\frac{\rd \bw}{\rd s}(s)\right|_{L^2} \leq \nu |A \bw(s)|_{L^2} + |B(\bw(s),\bu(s))|_{L^2} + |B(\bu(s),\bw(s))|_{L^2} + \\
 + |B(\bw(s),\bw(s))|_{L^2} + \beta |I_h(\bw(t_n) - \bw(t))|_{L^2} + \beta |I_h(\bw(t))|_{L^2} + \beta|\eta_n|_{L^2} \\
 \leq \nu|A\bw(s)|_{L^2} + |B(\bw(s),\bu(s))|_{L^2} + |B(\bu(s),\bw(s))|_{L^2} +  |B(\bw(s),\bw(s))|_{L^2} +\\
 + c_1 \beta \int_{t_n}^s \left| \frac{\rd \bw}{\rd \tau}(\tau)\right|_{L^2} \rd \tau + c_1 \beta |\bw(s)|_{L^2} + \beta \frac{E_1}{\lambda_1^{1/2}}. 
\end{multline}

Integrating with respect to $s$, between $t_n$ and $t \in [t_n, t_{n+1})$, yields
\begin{multline}\label{est7}
 \int_{t_n}^t \left|\frac{\rd \bw}{\rd s}(s)\right|_{L^2} \rd s \leq \int_{t_n}^t \left( \nu|A\bw(s)|_{L^2} + |B(\bw(s),\bu(s))|_{L^2} + |B(\bu(s),\bw(s))|_{L^2} + \right. \\
 \left. + |B(\bw(s),\bw(s))|_{L^2} + c_1 \beta |\bw(s)|_{L^2} + \beta \frac{E_1}{\lambda_1^{1/2}} \right) \rd s + c_1 \beta\kappa \int_{t_n}^t \left| \frac{\rd \bw}{\rd s}(s)\right|_{L^2} \rd s, 
\end{multline}
where we used an estimate similar to \eqref{doubleinttime}.

From the smallness condition \eqref{condkappag0} on $\kappa$, with $c \leq 1/(2c_1)$, we have in particular that $\beta \kappa \leq 1/(2c_1)$. It thus follows from \eqref{est7} that
\begin{multline}\label{est8}
 \int_{t_n}^t \left|\frac{\rd \bw}{\rd s}(s)\right|_{L^2} \rd s \leq c \int_{t_n}^t \left( \nu|A\bw(s)|_{L^2} + |B(\bw(s),\bu(s))|_{L^2} + |B(\bu(s),\bw(s))|_{L^2} + \right. \\
 \left. + |B(\bw(s),\bw(s))|_{L^2} + c_1 \beta |\bw(s)|_{L^2} + \beta \frac{E_1}{\lambda_1^{1/2}} \right) \rd s.
\end{multline}

Using \eqref{estnonlineartermHfunc}, we obtain
\be\label{est9}
|B(\bw,\bu)|_{L^2} \leq c_B M_1 \|\bw\|_{H^1} \left( 1 + \log \left( \frac{|A\bw|_{L^2}}{\lambda_1^{1/2} \|\bw\|_{H^1}} \right)\right)^{1/2},
\ee
\be\label{est10}
|B(\bw,\bw)|_{L^2} \leq c_B \|\bw\|_{H^1}^2 \left( 1 + \log \left( \frac{|A\bw|_{L^2}}{\lambda_1^{1/2} \|\bw\|_{H^1}} \right)\right)^{1/2}.
\ee
And using \eqref{estnonlineartermHL4L4L2}, we have
\be\label{est11}
|B(\bu,\bw)|_{L^2} \leq c_L^2 (M_0 M_1)^{1/2} \|\bw\|_{H^1}^{1/2} |A \bw|_{L^2}^{1/2}.
\ee

Now, applying H\"older inequality and using estimates \eqref{est9}-\eqref{est11}, it follows from \eqref{est8} that
\be\label{est12}
 \left( \int_{t_n}^t \left|\frac{\rd \bw}{\rd s}(s)\right|_{L^2} \rd s \right)^2 \leq c \kappa \int_{t_n}^t \varphi(s) \rd s + c \frac{(\kappa \beta E_1)^2}{\lambda_1},
\ee
where
\begin{multline}\label{defvarphiGeneral}
 \varphi(s) = \nu^2 |A\bw(s)|_{L^2}^2 + M_1^2 \|\bw(s)\|_{H^1}^2 \left( 1 + \log \left( \frac{|A\bw(s)|_{L^2}}{\lambda_1^{1/2} \|\bw(s)\|_{H^1}} \right)\right) + \\
 + M_0 M_1 \|\bw(s)\|_{H^1} |A\bw(s)|_{L^2} + \|\bw(s)\|_{H^1}^4 \left( 1 + \log \left( \frac{|A\bw(s)|_{L^2}}{\lambda_1^{1/2} \|\bw(s)\|_{H^1}} \right)\right) + \\
 + \beta^2 |\bw(s)|_{L^2}^2.
\end{multline}

Hence, using estimates \eqref{est1}-\eqref{est5} and \eqref{est12} into equation \eqref{energyeqwGeneral}, it follows that
\begin{multline}\label{energyineqdiffformGeneral}
 \frac{\rd}{\rd t} \|\bw(t)\|^2_{H^1} + \nu |A\bw(t)|^2_{L^2} \leq -\beta \|\bw(t)\|_{H^1}^2 \\
 + c M_1 \|\bw(t)\|_{H^1} |A\bw(t)|_{L^2} \left( 1 + \log\left( \frac{|A\bw(t)|_{L^2}}{\lambda_1^{1/2} \|\bw(t)\|_{H^1}} \right) \right)^{1/2} \\
 + c \|\bw(t)\|_{H^1}^2 |A\bw(t)|_{L^2} \left( 1 + \log\left( \frac{|A\bw(t)|_{L^2}}{\lambda_1^{1/2} \|\bw(t)\|_{H^1}} \right) \right)^{1/2} \\
 + \frac{c\beta^2\kappa}{\nu} \sum_{n=0}^\infty \chi_n(t) \int_{t_n}^t \varphi(s) \rd s + \beta \left(2 + \frac{c \beta^3 \kappa^2}{\nu\lambda_1} \right) E_1^2,
\end{multline}
for all $t \in [t_0,\infty)$.

Consider $R \geq 0$ given by
\[
 R = 2 M_1 + 3E_1.
\]

Since $\bw \in \mC([t_0,\infty);V)$ and 
\[
 \|\bw(t_0)\|_{H^1} \leq \|\bv(t_0)\|_{H^1} + \|\bu(t_0)\|_{H^1} \leq 2 M_1 \leq R,
\]
there exists $\tau \in (t_0,\infty)$ such that 
\be\label{boundVnormw}
 \|\bw(t)\|_{H^1} \leq 2R, \quad \forall t \in [t_0,\tau].
\ee
Define
\be\label{deftildetGeneral}
 \tilde{t} = \left\{ \tau \in [t_0,\infty) : \sup_{t\in [t_0,\tau]} \|\bw(t)\|_{H^1} \leq 2R \right\}.
\ee

Assume that $\tilde{t} < t_1$. 

Using the definitions of $R$ and $\tilde{t}$, we have that, for all $t \in [t_0, \tilde{t}]$,
\begin{multline}\label{absorbingterms}
 \frac{\nu}{2} |A\bw|_{L^2}^2 - c M_1 \|\bw\|_{H^1} |A\bw|_{L^2} \left( 1 + \log\left( \frac{|A\bw|_{L^2}}{\lambda_1^{1/2} \|\bw\|_{H^1}} \right) \right)^{1/2} \\
 - c \|\bw\|_{H^1}^2 |A\bw|_{L^2} \left( 1 + \log\left( \frac{|A\bw|_{L^2}}{\lambda_1^{1/2} \|\bw\|_{H^1}} \right) \right)^{1/2} + \frac{\beta}{2} \|\bw\|_{H^1}^2 \\
 \geq \frac{\nu \lambda_1}{2} \|\bw\|_{H^1}^2 \left[ \frac{|A\bw|_{L^2}^2}{\lambda_1 \|\bw\|_{H^1}^2} - c \frac{R}{\nu \lambda_1^{1/2}} \frac{|A\bw|_{L^2}}{\lambda_1^{1/2} \|\bw\|_{H^1}}  \left( 1 + \log\left( \frac{|A\bw|_{L^2}^2}{\lambda_1 \|\bw\|_{H^1}^2} \right) \right)^{1/2}\right.\\
 \left. + \frac{\beta}{\nu \lambda_1} \right]
\end{multline}

Define
\be
\phi(r) = r^2 - \rho r ( 1 + \log (r^2) )^{1/2} + B, \quad r \geq 1,
\ee
where 
\be\label{defrhoB}
\rho = c \frac{R}{\nu \lambda_1^{1/2}}, \quad B = \frac{\beta}{\nu \lambda_1}.
\ee

Note that 
\be\label{phi}
\phi(r) = \frac{r( \widetilde{\phi}(r^2) + B) + \rho(1 + \log (r^2))^{1/2}}{r + \rho(1 + \log (r^2))^{1/2}},
\ee
where 
\be\label{defphitilde}
\widetilde{\phi}(r) = r - {\rho}^2(1 + \log r).
\ee

It is not difficult to verify that
\be\label{minphi}
\min_{r\geq 1} \widetilde{\phi}(r) \geq -{\rho}^2 \log ({\rho}^2).
\ee

Thus, from \eqref{phi} and \eqref{minphi}, it follows that if
\be\label{condrhoB}
 B \geq {\rho}^2 \log ({\rho}^2),
\ee
then 
\be\label{phipositive}
\phi(r) \geq 0, \quad \forall r \geq 1.
\ee

Notice that, by the definition of $\rho$ and $B$ in \eqref{defrhoB}, \eqref{condrhoB} is valid due to hypothesis \eqref{condbetag0} on $\beta$. 

Hence, using \eqref{phipositive} with 
\be\label{defr}
r = \frac{|A\bw|_{L^2}}{\lambda_1^{1/2} \|\bw\|_{H^1}} \geq 1, 
\ee
we conclude that the expression on the right-hand side of \eqref{absorbingterms} is non-negative, for all $t \in [t_0, \tilde{t}]$. Using this fact in \eqref{energyineqdiffformGeneral}, we obtain that, for all $t \in [t_0, \tilde{t}]$,
\begin{multline}\label{energyineqdiffformGeneral2}
  \frac{\rd}{\rd t} \|\bw(t)\|^2_{H^1} + \frac{\nu}{2} |A\bw(t)|^2_{L^2} \leq -\frac{\beta}{2} \|\bw(t)\|_{H^1}^2  + \frac{c\beta^2\kappa}{\nu} \int_{t_0}^t \varphi(s) \rd s \\
+ \beta \left(2 + \frac{c \beta^3 \kappa^2}{\nu\lambda_1} \right) E_1^2.
\end{multline}

Integrating in time from $t_0$ to $t \in (t_0,\tilde{t}]$ and using the definition of $\varphi$ given in \eqref{defvarphiGeneral}, we obtain that
\begin{multline}\label{energyineqwintformGeneral1}
 \|\bw(t)\|_{H^1}^2 - \|\bw(t_0)\|_{H^1}^2 + \frac{\nu}{2} \int_{t_0}^t |A\bw(s)|_{L^2}^2 \rd s \leq \\
 \leq - \frac{\beta}{2} \int_{t_0}^t \|\bw(s)\|_{H^1}^2 \rd s + \frac{c\beta^2\kappa^2}{\nu} \int_{t_0}^t \varphi(s) \rd s
 + \beta\kappa \left(2 + \frac{c \beta^3 \kappa^2}{\nu\lambda_1} \right) E_1^2 \\
 \leq - \frac{\beta}{2} \int_{t_0}^t \|\bw(s)\|_{H^1}^2 \rd s + \frac{c\beta^2\kappa^2}{\nu} \int_{t_0}^t 
 \left[ \nu^2 |A\bw(s)|_{L^2}^2 + \right.\\
 M_1^2 \|\bw(s)\|_{H^1}^2 \left( 1 + \log \left( \frac{|A\bw(s)|_{L^2}}{\lambda_1^{1/2} \|\bw(s)\|_{H^1}} \right)\right) +  M_0 M_1 \|\bw(s)\|_{H^1} |A\bw(s)|_{L^2} \\
 \left. + \|\bw(s)\|_{H^1}^4 \left( 1 + \log \left( \frac{|A\bw(s)|_{L^2}}{\lambda_1^{1/2} \|\bw(s)\|_{H^1}} \right)\right) + \beta^2 |\bw(s)|_{L^2}^2 \right] \rd s \\
 + \beta\kappa \left(2 + \frac{c \beta^3 \kappa^2}{\nu\lambda_1} \right) E_1^2 .
\end{multline}

Using that $\beta\kappa \leq c$, due to \eqref{condkappag0}, we observe that
\begin{multline}\label{absorbingterms2}
 \frac{\nu}{4} \int_{t_0}^t |A\bw|_{L^2}^2 \rd s - \frac{c\beta^2\kappa^2}{\nu} \int_{t_0}^t \left[ M_1^2 \|\bw\|_{H^1}^2 \left( 1 + \log \left( \frac{|A\bw|_{L^2}}{\lambda_1^{1/2} \|\bw\|_{H^1}} \right)\right)  \right.\\
 \left. + \|\bw\|_{H^1}^4 \left( 1 + \log \left( \frac{|A\bw|_{L^2}}{\lambda_1^{1/2} \|\bw\|_{H^1}} \right)\right) \right] + \frac{\beta}{4}\int_{t_0}^t \|\bw\|_{H^1}^2 \rd s \\
 \geq \frac{\nu \lambda_1}{4} \int_{t_0}^t \|\bw\|_{H^1}^2\left[ \frac{|A\bw|_{L^2}^2}{\lambda_1 \|\bw\|_{H^1}^2} - \frac{c R^2}{\nu^2 \lambda_1} \left(1 + \log \left( \frac{|A\bw|_{L^2}^2}{\lambda_1 \|\bw\|_{H^1}^2} \right)\right) + \frac{\beta}{\nu \lambda_1} \right] \rd s
\end{multline}

Now using \eqref{defphitilde}-\eqref{phipositive}, we conclude that, due to condition \eqref{condbetag0} on $\beta$, the right-hand side of \eqref{absorbingterms2} is non-negative, for all $t \in [t_0, \tilde{t}]$. From \eqref{energyineqwintformGeneral1}, we then obtain
\begin{multline}
 \|\bw(t)\|_{H^1}^2 - \|\bw(t_0)\|_{H^1}^2 + \frac{\nu}{4} \int_{t_0}^t |A\bw(s)|_{L^2}^2 \rd s \leq - \frac{\beta}{4} \int_{t_0}^t \|\bw(s)\|_{H^1}^2 \rd s \\
 + \frac{c\beta^2\kappa^2}{\nu} \int_{t_0}^t \left( \nu^2 |A\bw(s)|_{L^2}^2 +  M_0 M_1 \|\bw(s)\|_{H^1} |A\bw(s)|_{L^2} + \beta^2 |\bw(s)|_{L^2}^2 \right) \rd s \\
 + \beta\kappa \left(2 + \frac{c \beta^3 \kappa^2}{\nu\lambda_1} \right) E_1^2.
\end{multline}

Using Young's inequality and Poincar\'e inequality \eqref{ineqPoincare}, we obtain, after some rearrangement of terms, that
\begin{multline}
 \|\bw(t)\|_{H^1}^2 - \|\bw(t_0)\|_{H^1}^2 + \frac{\nu}{4} ( 1 - c\beta^2 \kappa^2 ) \int_{t_0}^t |A\bw(s)|_{L^2}^2 \rd s \leq \\
 \leq -\frac{\beta}{4} \left( 1 - c \frac{\beta \kappa^2 (M_0 M_1)^2}{\nu^3} - c\frac{\beta^3 \kappa^2}{\nu \lambda_1} \right) \int_{t_0}^t \|\bw(s)\|_{H^1}^2 \rd s + \beta\kappa \left(2 + \frac{c \beta^3 \kappa^2}{\nu\lambda_1} \right) E_1^2.
\end{multline}
Now, from hypothesis \eqref{condkappag0}, it follows that
\be
 \|\bw(t)\|_{H^1}^2 - \|\bw(t_0)\|_{H^1}^2 + \frac{\nu}{8} \int_{t_0}^t |A\bw(s)|_{L^2}^2 \rd s \leq -\frac{\beta}{8} \int_{t_0}^t \|\bw(s)\|_{H^1}^2 \rd s + c E_1^2,
\ee
which implies in particular
\be\label{intAwbound}
\int_{t_0}^t |A\bw(s)|_{L^2}^2 \rd s \leq \frac{c}{\nu} \|\bw(t_0)\|_{H^1}^2 + \frac{c}{\nu} E_1^2.
\ee

Using Poincar\'e inequality \eqref{ineqPoincare} and the fact that $1+ \log r \leq r$, for all $r > 0$, in order to estimate the terms in the definition of $\varphi$ given in \eqref{defvarphiGeneral}, it follows from \eqref{energyineqdiffformGeneral2} that
\begin{multline}
  \frac{\rd}{\rd t} \|\bw(t)\|^2_{H^1} \leq -\frac{\beta}{2} \|\bw(t)\|_{H^1}^2 + \frac{c\beta^2\kappa}{\nu} \left( \nu^2 + \frac{R^2}{\lambda_1} + \frac{M_0 M_1}{\lambda_1^{1/2}} + \frac{\beta^2}{\lambda_1^2}\right) \int_{t_0}^t |A\bw(s)|_{L^2}^2 \rd s \\
  + \beta\left(2 + \frac{c \beta^3 \kappa^2}{\nu\lambda_1} \right) E_1^2 \\
  \leq  -\frac{\beta}{2} \|\bw(t)\|_{H^1}^2 + c\beta^2 \kappa \left(1 + \frac{M_0 M_1}{\nu^2\lambda_1^{1/2}} + \frac{R^2}{\nu^2\lambda_1} + \frac{\beta^2}{(\nu\lambda_1)^2} \right) (\|\bw(t_0)\|_{H^1}^2 + E_1^2)  \\
  + \beta\left(2+ \frac{c \beta^3 \kappa^2}{\nu\lambda_1} \right) E_1^2,
\end{multline}
where in the second inequality we used the bound from \eqref{intAwbound}.

Integrating in time from $t_0$ to $t \in [t_0, \tilde{t}]$, we have
\be\label{ineqwGronwallGeneral}
 \|\bw(t)\|_{H^1}^2 \leq \|\bw(t_0)\|_{H^1}^2\Exp^{-\frac{\beta}{2}(t-t_0)} + ( \gamma\|\bw(t_0)\|_{H^1}^2 + \sigma E_1^2 + 4 E_1^2) ( 1 - \Exp^{-\frac{\beta}{2}(t-t_0)}), 
\ee
where
\be
 \gamma = c\beta\kappa \left( 1 + \frac{M_0 M_1}{\nu^2 \lambda_1^{1/2}} + \frac{R^2}{\nu^2\lambda_1} + \frac{\beta^2}{(\nu\lambda_1)^2} \right)
\ee
and
\be
 \sigma = \gamma + \frac{c\beta^3 \kappa^2}{\nu\lambda_1}.
\ee

Since $\|\bw(t_0)\|_{H^1} \leq R$, we also have
\be\label{ineqwGronwallGeneral2}
 \|\bw(t)\|_{H^1}^2 \leq R^2\Exp^{-\frac{\beta}{2}(t-t_0)} + ( \gamma R^2 + \sigma E_1^2 + 4 E_1^2) ( 1 - \Exp^{-\frac{\beta}{2}(t-t_0)}).
\ee

Using condition \eqref{condkappag0} on $\kappa$, with a suitable absolute constant $c$, we obtain $\gamma + \sigma \leq 1/2$, which implies
\[
 \gamma R^2 + \sigma E_1^2 + 4 E_1^2 \leq \frac{R^2}{2} + 4 E_1^2 \leq R^2.
\]
Thus, it follows from \eqref{ineqwGronwallGeneral2} that
\[
 \|\bw(t)\|_{H^1} \leq R, \quad \forall t \in [t_0, \tilde{t}].
\]

In particular, $\|\bw(\tilde{t})\|_{H^1} \leq R$, and from the definition of $\tilde{t}$ in \eqref{deftildetGeneral} we conclude that, in fact, $\tilde{t} \geq t_1$. Therefore, we also have $\|\bw(t_1)\|_{H^1} \leq R$ and we can apply the same previous arguments to obtain that $\tilde{t} \geq t_2$ and $\|\bw(t_2)\|_{H^1} \leq R$. Proceeding inductively, we obtain that $\tilde{t} \geq t_n$, for all $n \geq 0$, so that, in fact,
\be\label{unifboundwH1}
 \|\bw(t)\|_{H^1} \leq R, \quad \forall t \in [t_0, \infty).
\ee
Also, analogously to \eqref{ineqwGronwallGeneral}, we obtain that
\be\label{ineqwGronwallGeneral3}
 \|\bw(t)\|_{H^1}^2 \leq \|\bw(t_n)\|_{H^1}^2\Exp^{-\frac{\beta}{2}(t-t_n)} + \left( \gamma \|\bw(t_n)\|_{H^1}^2 + \sigma E_1^2 + 4 E_1^2 \right) ( 1 - \Exp^{-\frac{\beta}{2}(t-t_n)}), 
\ee
for all $t \in [t_n,t_{n+1}]$ and for all $n \geq 0$. 

From \eqref{ineqwGronwallGeneral3}, it follows in particular that 
\[
 \|\bw(t_{n+1})\|_{H^1} \leq \theta \|\bw(t_n)\|_{H^1} +c E_1, \quad \forall n \geq 0,
\]
where
\be\label{defthetaGeneral}
 \theta = \left( \Exp^{-\frac{\beta\kappa}{2}} + \gamma ( 1 - \Exp^{-\frac{\beta\kappa}{2}})\right)^{1/2} < 1.
\ee

Thus,
\be\label{ineqwGronwallGeneral4}
 \|\bw(t_n)\|_{H^1} \leq \theta^n \|\bw(t_0)\|_{H^1} + c E_1 \sum_{j = 0}^{n-1} \theta^j, \quad \forall n \geq 1.
\ee

From \eqref{ineqwGronwallGeneral3} and \eqref{ineqwGronwallGeneral4}, it follows that
\be\label{ineqwGronwallGeneral5}
 \|\bw(t)\|_{H^1} \leq \theta^n \|\bw(t_0)\|_{H^1} + c E_1 \left( 1 + \sum_{j = 0}^{n-1} \theta^j \right), \quad \forall t \in [t_n,t_{n+1}], \forall n \geq 1.
\ee

Therefore,
\[
 \limsup_{t\rightarrow \infty} \|\bw(t)\|_{H^1} \leq c E_1.
\]

In particular, if $E_1 = 0$, it follows from \eqref{ineqwGronwallGeneral5} that 
\[
  \|\bw(t)\|_{H^1} \leq \theta^n \|\bw(t_0)\|_{H^1}, \quad \forall t \in [t_n,t_{n+1}], \forall n \geq 1,
\]
so that $\bw(t)$ converges exponentially to $0$ in $V$ as $t \rightarrow \infty$.
\end{proof}

% \begin{rmk}
%  Consider two solutions $\bv_1$ and $\bv_2$ of \eqref{DataAssEqDiscTimeGeneral} on $[t_0, \infty)$ with respect to the initial conditions $\bv_1(t_0) = \bv_0^1$ and $\bv_2(t_0) = \bv_0^2$, respectively, with $\bv_0^1, \bv_0^2 \in H$. Then, we have that $\bw = \bv_1 - \bv_2$ satisfies the following equation
%  \[
%   \frac{\rd \bw}{\rd t} + \nu A\bw + B(\bw,\bv_1) + B(\bv_2,\bw) = -\beta \sum_{n=0}^\infty I_h (\bw(t_n))\chi_n,
%  \]
%  Then, by using analogous arguments from the proof of Theorem \ref{thmasympestwGeneral}
% \end{rmk}

\begin{rmk}
 With the uniform bound obtained for $\bw = \bv - \bu$ in \eqref{unifboundwH1}, in the proof of Theorem \ref{thmasympestwGeneral}, we conclude that, under the hypotheses of this Theorem, the solution $\bv$ of the discrete data assimilation algorithm \eqref{DataAssEqDiscTimeGeneral} corresponding to an initial data $\bv_0 \in B_V(M_1)$ satisfies $\bv \in L^{\infty}(t_0,\infty;V)$, with
 \be\label{unifboundvH1}
  \|\bv(t)\|_{H^1} \leq 3 (M_1 + E_1), \quad \forall t \in [t_0, \infty).
 \ee
\end{rmk}

\section{Stationary Statistical Analysis}\label{secStatAnal}

Most of the usual applications of data assimilation are given in the context of fully developed turbulent flows, the classical example being the atmosphere. In such flows, the instantaneous physical quantities display an unpredictable and erratic behavior in time, while the averages (in space, in time or with respect to an ensemble of experiments) of such quantities behave in a more regular way. For this reason, it is common for experimentalists to obtain measurements of such flows given by averages of the associated physical quantities.

Therefore, it is useful to obtain a relation between averages of physical quantities associated to a reference solution of the underlying evolution system and averages of the same quantities associated to the corresponding approximating solution, given by the data assimilation algorithm.

In the following result, we obtain an estimate of the difference between time averages of physical quantities associated to a solution $\bu$ of \eqref{functionaleq} and the corresponding solution $\bv$ of \eqref{DataAssEqDiscTimeGeneral} with respect to an initial data $\bv_0 \in V$. The physical quantity is represented by a function $\Phi: H \to \mathbb{R}$, which is assumed to be a Lipschitz continuous function, with respect to the $L^2$ norm, when restricted to a suitable ball in $V$. The proof follows from the estimate of the difference between $\bv$ and $\bu$ obtained in Theorem \ref{thmasympestwGeneral}.

\begin{thm}\label{thmasympestdifftimeav}
  Under the hypotheses of Theorem \ref{thmasympestwGeneral}, for every function $\Phi: H \rightarrow \mathbb{R}$ such that its restriction to the ball $B_V(3( M_1 + E_1))$ is Lipschitz continuous with respect to the $L^2$ norm, with Lipschitz constant $L_{\Phi} > 0$, it holds
 \be\label{asympestdifftimeav}
  \limsup_{T \rightarrow \infty} \left| \frac{1}{T}\int_{t_0}^{t_0 + T} \Phi(\bv(t)) \rd t - \frac{1}{T}\int_{t_0}^{t_0 + T} \Phi(\bu(t)) \rd t \right| \leq c\frac{L_{\Phi} E_1}{\lambda_1^{1/2}}.
 \ee
\end{thm}
\begin{proof}
 Denote $\bw = \bv - \bu$ and let $T \geq t_0$. From \eqref{unifboundsattractor} and \eqref{unifboundvH1}, we have that $\bv(t), \bu(t) \in B_V(3( M_1 + E_1))$, for all $t \in [t_0, \infty)$. Therefore, since the restriction of $\Phi$ to $B_V(3( M_1 + E_1))$ is a Lipschitz continuous function with Lipschitz constant $L_{\Phi}$, we obtain that
 \begin{multline}\label{ineqdifftimeav1}
  \left| \frac{1}{T}\int_{t_0}^{t_0 + T} \Phi(\bv(t)) \rd t - \frac{1}{T}\int_{t_0}^{t_0 + T} \Phi(\bu(t)) \rd t \right| \leq \\
  \leq \frac{L_{\Phi}}{T} \int_{t_0}^{t_0 + T} | \bw(t)|_{L^2} \rd t \leq \frac{L_{\Phi}}{\lambda_1^{1/2} T} \int_{t_0}^{t_0 + T} \| \bw(t)\|_{H^1} \rd t,
 \end{multline}
where in the last estimate we used Poincar\'e inequality \eqref{ineqPoincare}.

From inequality \eqref{ineqwGronwallGeneral3} in the proof of Theorem \ref{thmasympestwGeneral}, we have in particular that
\be
 \|\bw(t)\|_{H^1} \leq \|\bw(t_0)\|_{H^1} + c E_1, \quad \forall t \in [t_0, t_1].
\ee

Also, recall inequality \eqref{ineqwGronwallGeneral5}, given by
\be\label{ineqwGronwallGeneral5recall}
 \|\bw(t)\|_{H^1} \leq \theta^n \|\bw(t_0)\|_{H^1} + c E_1 \left( 1 + \sum_{j = 0}^{n-1} \theta^j \right), \quad \forall t \in [t_n,t_{n+1}], \forall n \geq 1,
\ee
with $\theta < 1$ given in \eqref{defthetaGeneral}.

Define
\[
 N = \max \{ n \in \mathbb{N}: t_n = t_0 + \kappa n \leq t_0 + T \}.
\]

Then, using \eqref{ineqdifftimeav1}-\eqref{ineqwGronwallGeneral5recall}, we obtain that
\begin{multline}\label{ineqdifftimeav2}
  \left| \frac{1}{T}\int_{t_0}^{t_0 + T} \Phi(\bv(t)) \rd t - \frac{1}{T}\int_{t_0}^{t_0 + T} \Phi(\bu(t)) \rd t \right| \leq \\
  \leq \frac{L_{\Phi}}{\lambda_1^{1/2} T} \left(\int_{t_N}^{t_0 + T} \|\bw(t)\|_{H^1} \rd t + \sum_{n=0}^{N-1} \int_{t_n}^{t_{n+1}} \|\bw(t)\|_{H^1} \rd t \right)\\
  \leq \frac{L_{\Phi}}{\lambda_1^{1/2} T} \sum_{n=0}^N  \int_{t_n}^{t_{n+1}} \|\bw(t)\|_{H^1} \rd t \\
  \leq \frac{L_{\Phi}\kappa}{\lambda_1^{1/2} T} \left[  \|\bw(t_0)\|_{H^1} + c E_1 + \right.\\
  \left. + \sum_{n=1}^N \left( \theta^n \|\bw(t_0)\|_{H^1} + c E_1 \left( 1 + \sum_{j=0}^{n-1} \theta^j \right)\right)\right] \\
  \leq \frac{L_{\Phi}\kappa}{\lambda_1^{1/2} T} \left[ \|\bw(t_0)\|_{H^1} \sum_{n=0}^N \theta^n + c E_1 \left( 1 + N + N \sum_{j=0}^{N-1} \theta^j \right)\right]\\
  \leq \frac{L_{\Phi}\kappa}{\lambda_1^{1/2} T} \|\bw(t_0)\|_{H^1} \sum_{n=0}^N \theta^n + c \frac{L_{\Phi}\nu E_1}{T \lambda_1^{1/2}} \left( \kappa + T + T\sum_{j=0}^{N-1} \theta^j \right),
\end{multline}
where in the last inequality we used that $\kappa N \leq T$. Inequality \eqref{asympestdifftimeav} then follows by taking the $\limsup$ as $T \rightarrow \infty$ in \eqref{ineqdifftimeav2}.
\end{proof}

Next, we show that, in the particular case when $E_1 = 0$, we can obtain a relation between ensemble averages associated to the reference solution $\bu$ of \eqref{functionaleq} and time averages associated to the corresponding solution $\bv$ of \eqref{DataAssEqDiscTimeGeneral}, with respect to an initial data $\bv_0 \in V$. This is motivated by the relation between time averages and ensemble averages already known for the Navier-Stokes equations, which is proved as in the Krylov-Bogolyubov theory \cite{KrylovBogolyubov1937}. 

For this purpose, we must recall the notion of Banach generalized limits (see, e.g., \cite[Section IV.1.3]{FMRT2001}). These are linear functionals denoted by $\LIM_{T\to \infty}$ and defined on the space of bounded real-valued functions on $[t_0,\infty)$, $\mB([t_0,\infty))$, which coincide with the classical limit whenever the latter exists. Moreover, it satisfies the following property
\be\label{propLIM}
 \liminf_{T\to\infty} g(T) \leq  \underset{T\to \infty}{\LIM} g(T) \leq \limsup_{T\to\infty} g(T), \quad \forall g \in \mB([t_0,\infty)).
\ee

Let us denote by $\mC(H)$ the space of continuous real-valued functions on $H$. Recall that $\{S(t)\}_{t\geq t_0}$, given in \eqref{defsemigroup}, denotes the semigroup associated to the 2D Navier-Stokes equations on $[t_0,\infty)$. We then have the following known result. 

\begin{thm}\label{thmEqTimeAvEnsAvNSE}
 Let $\bu_0 \in \mA$. Then, there exists an invariant probability measure $\mu_{\bu_0}$ with respect to $\{S(t)\}_{t\geq t_0}$ such that, for all $\Phi \in \mC(H)$,
 \[
   \underset{T\to \infty}{\LIM} \frac{1}{T}\int_{t_0}^{t_0+T} \Phi(S(t) \bu_0) \rd t = \int_H \Phi(\boldsymbol{\xi}) \rd \mu_{\bu_0}(\boldsymbol{\xi}).
 \]
\end{thm}

The proof of Theorem \ref{thmEqTimeAvEnsAvNSE} can be found in \cite[Chapter 4]{FMRT2001}.

We now have the result showing the equality between ensemble averages of physical quantities associated to the reference solution and time averages of the same quantities associated to the corresponding approximating solution. Notice that here the function $\Phi$ representing the physical quantity belongs to $\mC(H)$.

\begin{thm}\label{thmEqualityTimeAvEnsembleAv}
  Under the hypotheses of Theorem \ref{thmasympestwGeneral}, denote $\bu_0 = \bu(t_0)$ and assume that $E_1 = 0$. Then, given a generalized limit $\LIM_{T\to \infty}$, there exists an invariant probability measure $\mu_{\bu_0}$ with respect to $\{S(t)\}_{t\geq t_0}$ such that 
 \[
  \underset{T\to \infty}{\LIM} \frac{1}{T} \int_{t_0}^{t_0 + T} \Phi(\bv(t)) \rd t = \int_H \Phi(\boldsymbol{\xi}) \rd \mu_{\bu_0}(\boldsymbol{\xi}), \quad \forall \Phi \in \mC(H).
 \]
\end{thm}
\begin{proof}
Since $\bv(t), \bu(t) \in B_V(3( M_1 + E_1))$, for all $t \in [t_0, \infty)$, and $E_1 = 0$, from Theorem \ref{thmasympestdifftimeav} it follows that, for every function $\Phi : H \to \mathbb{R}$ such that its restriction to $B_V(3( M_1 + E_1))$ is a Lipschitz continuous function, we have
\be\label{TimeAvEqual}
 \limsup_{T\to\infty} \left| \frac{1}{T}\int_{t_0}^ {t_0+T} \Phi(\bv(t)) \rd t - \frac{1}{T}\int_{t_0}^ {t_0+T} \Phi(\bu(t)) \rd t\right| = 0.
\ee
But since $B_V(3( M_1 + E_1))$ is compact in $H$, by the Stone-Weierstrass Theorem, the set of Lipschitz continuous functions on $B_V(3( M_1 + E_1))$, with respect to the norm topology in $H$, is dense in the space of continuous functions on $B_V(3( M_1 + E_1))$. Therefore, \eqref{TimeAvEqual} is also valid for every $\Phi \in \mC(H)$. Using \eqref{propLIM}, this implies that
\[
 \underset{T\to\infty}{\LIM} \left( \frac{1}{T}\int_{t_0}^ {t_0+T} \Phi(\bv(t)) \rd t - \frac{1}{T}\int_{t_0}^ {t_0+T} \Phi(\bu(t)) \rd t\right) = 0, \quad \forall \Phi \in \mC(H).
\]
Thus, from Theorem \ref{thmEqTimeAvEnsAvNSE}, we obtain that
\begin{multline}\label{eqTimeAvEnsembleAvVMuU}
  \underset{T\to\infty}{\LIM} \frac{1}{T}\int_{t_0}^ {t_0+T} \Phi(\bv(t)) \rd t = \underset{T\to\infty}{\LIM}\frac{1}{T}\int_{t_0}^ {t_0+T} \Phi(\bu(t)) \rd t = \int_H \Phi(\boldsymbol{\xi}) \rd \mu_{\bu_0}(\boldsymbol{\xi}),\\
  \forall \Phi \in \mC(H).
\end{multline}

\end{proof}

The result of Theorem \ref{thmEqualityTimeAvEnsembleAv} above provide us with a way of obtaining information on the averages of physical quantities associated to the unknown reference solution through time averages of the same quantities associated to the approximating solution, which is computed by using our discrete data assimilation algorithm.

\section{Acknowledgements}

The work of C. F. was supported by the ONR grant N00014-15-1-2333. The work of C. F. M. was supported by CNPq, Conselho Nacional de Desenvolvimento Científico e Tecnológico - Brazil, under the grant 203552/2014-8. The work of E. S. T. was supported in part by the ONR grant N00014-15-1-2333 and the NSF grants DMS-1109640 and DMS-1109645.

\setcounter{equation}{0}
\appendix

\section*{Appendix}
\renewcommand{\thesection}{A}
\renewcommand{\theequation}{{A.}\arabic{equation}}

In this appendix, we consider a concrete and physically relevant example of an interpolant operator $I_h$ given by local averages over finite volume elements (see, e.g., \cite{FoiasTiti1991,JonesTiti1992,JonesTiti1993}). Our purpose is to show explicit uniform estimates of the $L^2(\Omega)^2$ and $H^1(\Omega)^2$ norms for the sequence of errors $\{\eta_n\}_{n\in\mathbb{N}}$, which was introduced in section \ref{secDiscDataAssAlg}. With this we intend to show, in particular, that the inequality \eqref{bounderrorV} considered in Theorem \ref{thmasympestwGeneral} is a natural condition on such errors, that can be effectively verified in applications.

Let $\Omega \subset \mathbb{R}^2$ be an open and bounded set. Let us consider a spatial mesh with resolution of size $h >0$, given by a family $\{\mU_j\}_{1\leq j \leq N_h}$ of disjoint sets in $\mathbb{R}^2$ with strictly positive measure $|\mU_j|$, such that $\Omega = \bigcup_{j=1}^{N_h} \mU_j$ and such that the diameter $\textnormal{diam}(\mU_j)$ of each $\mU_j$ satisfies $\textnormal{diam}(\mU_j) \leq h$.

Now, for each $\mU_j$, $j = 1, \ldots, N_h$, let $\widetilde{\mU}_j$ be an open set with $\mU_j \subset \widetilde{\mU}_j$ and such that $\textnormal{diam}(\widetilde{\mU}_j) < 2h$. Then, let $\{\Psi_j\}_{1\leq j \leq N_h}$ be a partition of unity subordinated to  $\{\widetilde{\mU}_j\}_{1\leq j \leq N_h}$, with each $\Psi_j$ being a smooth and compactly supported real-valued function on $\mathbb{R}^2$ such that $\textnormal{supp}\Psi_j \subset \widetilde{\mU}_j$, 
\be\label{boundPsij}
 0 \leq \Psi_j(x) \leq 1, \quad \forall x \in \mathbb{R}^2,
\ee
\be\label{sumPsijEquals1}
 \sum_{j=1}^{N_h} \Psi_j(x) = 1, \quad \forall x \in \Omega,
\ee
and
\be\label{boundgradPsij}
 |\nabla \Psi_j(x)| \leq \frac{C_0}{h}, \quad \forall x \in \mathbb{R}^2,
\ee
where $C_0$ is a constant.

Note that, for all $j = 1, \ldots,N_h$, we have that the measure $|\widetilde{\mU}_j|$ of $\widetilde{\mU}_j$ satisfies $|\widetilde{\mU}_j| \leq 4h^2$, so that
\[
 \sum_{j=1}^{N_h} |\widetilde{\mU}_j| \leq 4 N_h h^2.
\]
We assume that $h$ is small enough so that 
\be\label{assumptionHDelta}
 4 N_h h^2 \leq 2 |\Omega|,
\ee
where $|\Omega|$ denotes the measure of $\Omega$.

For each $k \in \mathbb{N}$, let $\widetilde{\mU}_{j_1}, \ldots,\widetilde{\mU}_{j_{n_k}}$ be all the sets from the family $\{\widetilde{\mU}_j\}_{1 \leq j \leq N_h}$ which have a non-empty intersection with $\widetilde{\mU}_k$, i.e.
\be\label{UkUjiNonempty}
 \widetilde{\mU}_k \cap \widetilde{\mU}_{j_i} \neq \emptyset, \quad \forall i=1,\ldots,n_k,
\ee
and
\be\label{UkUjempty}
 \widetilde{\mU}_k \cap \widetilde{\mU}_j = \emptyset, \quad \forall j \notin \{j_1,\ldots,j_{n_k}\}.
\ee
We assume that there exists a constant $C_1$, which is independent of $h$, such that 
\be\label{boundNk}
 n_k \leq C_1, \quad \forall k \in \mathbb{N}.
\ee

\begin{rmk}
 In the case $\Omega = (0,L) \times (0,L)$, $L>0$, one can think of the spatial mesh for $\Omega$ as, for example, a square mesh grid, with each $\mU_j$ being a square with diameter $h$; $\Psi_j$ being the mollification of the characteristic function of $\mU_j$; and $\widetilde{\mU}_j$ being an open square with diameter $h+ \delta$, with $\delta < h$ depending on the mollification parameter (see, e.g., \cite[Appendix A]{Azouanietal2014} for more details).

The reason why we consider smooth functions instead of the characteristic functions of the sets $\mU_j$ is that we want to obtain, in particular, an estimate of $\eta_n$, given explicitly below in \eqref{eta_n_appendix}, with respect to the norm in $H^1(\Omega)^2$, which involves the gradient of $\eta_n$.
\end{rmk}

For each $j= 1,\ldots,N_h$, we denote by $\psi_j$ the restriction of the function $\Psi_j$ to the set $\Omega$. Note that, in particular, relations \eqref{boundPsij}-\eqref{boundgradPsij} are also satisfied by $\psi_j$, for all $x \in \Omega$.

Now, consider the interpolant operator $I_h: L^2(\Omega)^2 \to L^2(\Omega)^2$ given by
\be\label{defIhAppendix}
 I_h(\varphi) = \sum_{j=1}^{N_h} \overline{\varphi}_j \psi_j, \quad \forall \varphi \in L^2(\Omega)^2,
\ee
where 
\be
 \overline{\varphi}_j = \frac{1}{|\mU_j|}\int_{\mU_j} \varphi(x) \rd x
\ee
is the local average of $\varphi$ over the volume element $\mU_j$.

We consider the measurement over each volume element $\mU_j$, $j=1,\ldots,N_h$, at time $t_n$, given by
%Let us consider the measurements associated to the reference solution $\bu$ given at each subset $\mU_j$, $j=1,\ldots,N_h$, of the spatial mesh at time $t_n$, given by
\be
 \overline{\bu}_j(t_n) + \varepsilon_{n,j},
\ee
where 
\be
 \overline{\bu}_j(t_n) = \frac{1}{|\mU_j|}\int_{\mU_j} \bu(t_n,x) \rd x
\ee
is the local average over $\mU_j$ of the exact value of the reference solution $\bu$ at time $t_n$, and $\varepsilon_{n,j}$ is the corresponding error vector in the measurement. Note that the values $\overline{\bu}_j(t_n)$ and $\varepsilon_{n,j}$ are not known separately, but only their sum. However, an estimate of the maximum size of the errors $\varepsilon_{n,j}$, $n\in \mathbb{N}$, $j = 1,\ldots,N_h$, is usually given in terms of the accuracy associated to the devices used in an experiment. We denote the maximum error size by $\varepsilon > 0$, so that
%and $\varepsilon_{n,j} > 0$ represents the error corresponding to the measurements on $\mU_j$ at time $t_n$. We assume that there exists a constant $\varepsilon > 0$ such that 
\be\label{boundErrorsEpsilon_nj}
 |\varepsilon_{n,j}| \leq \varepsilon, \quad \forall n \in \mathbb{N}, \forall j = 1, \ldots, N_h.
\ee

Then, using the interpolant operator $I_h$ based on these values, we obtain, for each $n \in \mathbb{N}$, a vector field $\tilde{\bu}(t_n)$ defined on $\Omega$, given by
\[
 \tilde{\bu}(t_n) = \sum_{j=1}^{N_h} \overline{\bu}_j(t_n) \psi_j + \sum_{j=1}^{N_h} \varepsilon_{n,j} \psi_j.
\]
Comparing with the definition of $\tilde{\bu}(t_n)$ given in \eqref{repmeasurements}, we then have
\be
 I_h(\bu(t_n)) = \sum_{j=1}^{N_h} \overline{\bu}_j(t_n) \psi_j, 
\ee
and 
\be\label{eta_n_appendix}
 \eta_n = \sum_{j=1}^{N_h} \varepsilon_{n,j} \psi_j.
\ee

In the following proposition, we obtain estimates for the error term $\eta_n$ given in \eqref{eta_n_appendix}.

\begin{prop}\label{propAppendix}
 For all $n \in \mathbb{N}$, the following inequalities hold
 \begin{enumerate}[(i)]
  \item\label{propAppendix_i}  $\displaystyle |\eta_n|_{L^2} \leq \varepsilon |\Omega|^{1/2}$;
  \vspace{0.2cm}
  \item\label{propAppendix_ii} $\displaystyle \|\eta_n\|_{H^1} \leq c\frac{\varepsilon}{h} |\Omega|^{1/2}$,
  where $c = C_0 (2C_1)^{1/2}$, with $C_0$ and $C_1$ being the constants from \eqref{boundgradPsij} and \eqref{boundNk}, respectively.
 \end{enumerate}
\end{prop}
\begin{proof}
 Using \eqref{sumPsijEquals1} and \eqref{boundErrorsEpsilon_nj}, we obtain that
 \begin{multline*}
  |\eta_n|_{L^2}^2 = \left| \int_{\Omega} \left( \sum_{k=1}^{N_h} \varepsilon_{n,k} \psi_k(x) \right) \left( \sum_{j=1}^{N_h} \varepsilon_{n,j} \psi_j(x) \right) \rd x \right| \\
  \leq \varepsilon^2 \int_{\Omega} \left( \sum_{k=1}^{N_h} \psi_k(x) \right) \left( \sum_{j=1}^{N_h} \psi_j(x) \right) \rd x = \varepsilon^2 |\Omega|,
 \end{multline*}
 which proves \eqref{propAppendix_i}.
 
 Now, for the proof of \eqref{propAppendix_ii}, note that
 \begin{multline*}
  \|\eta_n\|_{H^1}^2 = \left| \left( \sum_{k=1}^{N_h} \varepsilon_{n,k} \nabla \psi_k(x) \right) \cdot  \left( \sum_{j=1}^{N_h} \varepsilon_{n,j} \nabla \psi_j(x) \right) \right| \\
  \leq \varepsilon^2 \sum_{k=1}^{N_h} \sum_{j=1}^{N_h} \int_{\widetilde{\mU}_k \cap \widetilde{\mU}_j \cap \Omega} |\nabla \psi_k(x)| |\nabla \psi_j(x)| \rd x.
 \end{multline*}
 Then, using \eqref{UkUjiNonempty}, \eqref{UkUjempty} and H\"older inequality, we obtain
 \[
  \|\eta_n\|_{H^1}^2 \leq \varepsilon^2 \sum_{k=1}^{N_h} \sum_{i=1}^{n_k} \left( \int_{\widetilde{\mU}_k \cap \Omega} |\nabla \psi_k(x)|^2 \rd x \right)^{1/2}  \left( \int_{\widetilde{\mU}_{j_i} \cap \Omega} |\nabla \psi_{j_i}(x)|^2 \rd x \right)^{1/2}.
 \]

 From \eqref{boundgradPsij}, \eqref{assumptionHDelta} and \eqref{boundNk}, it follows that
 \[
  \|\eta_n\|_{H^1}^2 \leq \varepsilon^2 \frac{C_0^2}{h^2} \sum_{k=1}^{N_h} \sum_{i=1}^{n_k} |\widetilde{\mU}_k|^{1/2} |\widetilde{\mU}_{j_i}|^{1/2} 
  \leq \varepsilon^2 \frac{C_0^2}{h^2} C_1 4 N_h h^2 \leq \varepsilon^2 \frac{C_0^2}{h^2} C_1 2 |\Omega|,
 \]
 which proves \eqref{propAppendix_ii}.
\end{proof}

Comparing inequality \eqref{propAppendix_ii} of Proposition \ref{propAppendix} with the condition \eqref{bounderrorV} on the sequence $\{\eta_n\}_{n\in\mathbb{N}}$ given in Theorem \ref{thmasympestwGeneral}, we see that $E_1$ in this particular case of $I_h$ defined by \eqref{defIhAppendix}, is given by
\be
 E_1 =  C_0 (2C_1)^{1/2} \frac{\varepsilon}{h} |\Omega|^{1/2},
\ee
where $C_0$ and $C_1$ are the constants from \eqref{boundgradPsij} and \eqref{boundNk}, respectively.

\end{document}